\documentclass[12pt,a4paper,twoside]{article}

\usepackage{fancyhdr,graphicx,subfigure,psfrag}
\usepackage{amsmath,amsfonts,amsthm}
\usepackage[abs]{overpic}
\usepackage{epsfig}

 \DeclareMathOperator{\diag}{diag}

\newcommand{\D}{\mathrm{d}}
\newcommand{\p}{\partial}

\newcommand{\tr}{\mathrm{tr}}

\newtheorem{theorem}{Theorem}
\newtheorem{lemma}{Lemma}
\newtheorem{corollary}{Corollary}
\newtheorem{remark}{Remark}
\newtheorem{proposition}{Proposition}
\newtheorem{definition}{Definition}

\numberwithin{equation}{section}
\begin{document}
\title{The rank 1 real Wishart spiked model I. Finite $N$ analysis}
\author{M. Y. Mo}
\date{}
\maketitle
\begin{abstract}
This is the first part of a paper that studies the phase transition
in the asymptotic limit of the rank 1 real Wishart spiked model. In
this paper, we consider $N$-dimensional real Wishart matrices $S$ in
the class $W_{\mathbb{R}}\left(\Sigma,M\right)$ in which all but one
eigenvalues of $\Sigma$ is $1$. Let the non-trivial eigenvalue of
$\Sigma$ be $1+\tau$, then as $N$, $M\rightarrow\infty$, with
$N/M=\gamma^2$ finite and non-zero, the eigenvalue distribution of
$S$ will converge into the Machenko-Pastur distribution inside a
bulk region. As $\tau$ increases from zero, one starts seeing stray
eigenvalues of $S$ outside of the support of the Machenko-Pastur
density. As the first of these stray eigenvalues leaves the bulk
region, a phase transition will occur in the largest eigenvalue
distribution of the Wishart matrix. In this paper will compute the
asymptotics of the largest eigenvalue distribution when the phase
transition occur. In the this first half of the paper, we will
establish the results that are valid for all $N$ and $M$ and will
use them to carry out the asymptotic analysis in the second half of
the paper, which will follow shortly. In particular, we have derived
a formula for the integral $\int_{O(N)}e^{-\tr(XgYg^T)}g^T\D g$ when
$X$, $Y$ are symmetric and $Y$ is a rank 1 matrix. This allows us to
write down a Fredholm determinant formula for the largest eigenvalue
distribution and analyze it using orthogonal polynomial techniques.
This approach is very different from a recent paper \cite{BB}, in
which the largest eigenvalue distribution was obtained using
stochastic operator method.
\end{abstract}
\section{Introduction}
Let $X$ be an $N\times M$ (throughout the paper, we will assume $M>
N$ and $N$ is even) matrix such that each column of $X$ is an
independent, identical $N$-variate random variable with normal
distribution and zero mean. Let $\Sigma$ be its covariance matrix,
i.e. $\Sigma_{ij}=E(X_{i1}X_{j1})$. Then $\Sigma$ is an $N\times N$
positive definite symmetric matrix and we denote its eigenvalues by
$1+\tau_j$. The matrix $S$ defined by $S=\frac{1}{M}XX^T$ is a real
Wishart matrix in the class $W_{\mathbb{R}}\left(\Sigma,M\right)$.
We can think of each column of $X$ as a draw from a $N$-variate
random variable with the normal distribution and zero mean, then $S$
is the the sample covariance matrix for the samples represented by
$X$. Real Wishart matrices are good models of sample covariance
matrices in many situations and have applications in many areas such
as finance, genetic studies and climate data. (See \cite{johnstone}
for example.)

In many of these applications, one has to deal with data in which
both $N$ and $M$ are large, while the ratio $N/M$ is finite and
non-zero. In particular, in applications to principle analysis, one
would like to study the asymptotic behavior of the largest
eigenvalue of $S$ as $N$, $M\rightarrow\infty$ with
$M/N\rightarrow\gamma^2\geq 1$ fixed.

For many statistical data, it was noted in \cite{johnstone} that in
the asymptotic limit, the eigenvalue distribution of the sample
covariance matrix will converge to a distribution whose density is
given by the Machenko-Pastur law \cite{MP} inside a bulk region (See
\cite{Bai},\cite{BS}.)
\begin{equation}\label{eq:MP}
\rho(\lambda)=\frac{\gamma}{2\pi\lambda}\sqrt{(\lambda-b_-)(b_+-\lambda)}\chi_{[b_-,b_+]},
\end{equation}
where $\chi_{[b_-,b_+]}$ is the characteristic function for the
interval $[b_-,b_+]$ and $b_{\pm}=(1\pm \gamma^{-1})^2$. However,
outside of the bulk region, there are often a finite number of large
eigenvalues at isolated locations. This behavior prompted the
introduction of the spiked model in \cite{johnstone}, which are
Wishart matrices with a covariance matrix with all but a finite
number of eigenvalues that are not equal to one. These non-trivial
eigenvalues in the covariance matrix will then be responsible for
the spikes that appear in the eigenvalue distribution of the sample
covariance matrix. The number of these non-trivial eigenvalues in
$\Sigma$ is called the rank of the spiked model.

Of particular interest is a phase transition that arises in the
largest eigenvalue distributions when the first of these spikes
starts leaving the bulk region. This phenomenon was first studied in
\cite{baik04} for the complex Wishart spiked model and then in
\cite{Wang} for the rank 1 quarternionic Wishart spiked model.
Despite having the most applications, the asymptotics for real
Wishart spiked model has not been solved until very recently
\cite{BB}. The main goal of this paper is to obtain the largest
eigenvalue distribution for the rank 1 real Wishart spiked model in
the asymptotic limit. In a recent paper \cite{BB}, the asymptotic
largest eigenvalue distribution for the rank 1 real Wishart ensemble
was obtained by using a completely different approach to ours. In
\cite{BB}, the authors first use the Housefolder algorithm to reduce
a Wishart matrix into tridiagonal form. Such tridiagonal matrix is
then treated as a discrete random Schr\"odinger operator and by
taking an appropriate scaling limit, the authors obtained a
continuous random Schr\"odinger operator on the half-line. By doing
so, the authors in \cite{BB} bypass the problem of determining the
eigenvalue j.p.d.f. for the real Wishart ensemble and obtain the
largest eigenvalue distribution in the asymptotic limit.

On the other hand, the approach presented in this paper uses
orthogonal polynomial techniques that are closer to those in
\cite{baik04} and \cite{Wang}. We will now outline our method.

One of the main difficulties in the asymptotic analysis of the real
Wishart ensembles is to find a simple expression for the j.p.d.f. of
its eigenvalues. Let $\lambda_j$ be the eigenvalues of the Wishart
matrix, then the j.p.d.f. for the real Wishart ensemble is given by
\begin{equation}\label{eq:jpdf}
P(\lambda)=\frac{1}{Z_{M,N}}|\Delta(\lambda)|\prod_{j=1}^{N}\lambda_j^{\frac{M-N-1}{2}}\int_{O(N)}
e^{-\frac{M}{2}\tr(\Sigma^{-1}gS g^{-1})}g^Tdg,
\end{equation}
where $g^TdG$ is the Haar measure on $O(N)$ and $Z_{M,N}$ is a
normalization constant. The expression of the j.p.d.f. for the
complex and quarternionic Wishart ensembles are similar. In the
complex case, the integral in the j.p.d.f. will be over the unitary
group while in the quarternionic case, the integral will be over the
symplectic group. One of the main difficulties in the asymptotic
analysis of Wishart ensembles is to evaluate the integral in
(\ref{eq:jpdf}). In the complex case, this integral can be evaluated
using the Harish-Chandra \cite{HC} (or Itzykson Zuber \cite{IZ})
formula, while in the quarternionic case, the integral can be
written as an infinite series in terms of Zonal polynomials and such
series converges to a simple function in the rank 1 case. For the
real case, however, the Harish-Chandra Itzykson Zuber formula does
not apply and while the series expression in terms of Zonal
polynomials still exists, such series expression do not seem to
converge into a simple function. In fact, our first result is that
the integral over $O(N)$ in (\ref{eq:jpdf}) is a hyper-elliptic
integral in the eigenvalues $\lambda_1,\ldots,\lambda_N$.
\begin{theorem}\label{thm:main1}
Assuming $N$ is even. Let the non-trivial eigenvalue in the
covariance matrix $\Sigma$ be $1+\tau $. Then the j.p.d.f. of the
eigenvalues in the rank 1 real Wishart spiked model with covariance
matrix $\Sigma$ is given by
\begin{equation}\label{eq:main1}
P(\lambda)=\tilde{Z}_{M,N}^{-1}
\int_{\Gamma}|\Delta(\lambda)|e^{Mt}\prod_{j=1}^Ne^{-\frac{M}{2}\lambda_j}\lambda_j^{\frac{M-N-1}{2}}\left(t-\frac{\tau
}{2(1+\tau )}\lambda_j\right)^{-\frac{1}{2}}dt,
\end{equation}
where $\Gamma$ is a contour that encloses all the points $\frac{\tau
}{2(1+\tau )}\lambda_1,\ldots,\frac{\tau }{2(1+\tau )}\lambda_N$
that is oriented in the counter-clockwise direction and
$\tilde{Z}_{M,N}$ is the normalization constant. The branch cuts of
the square root $\left(t-\frac{\tau }{2(\tau
+1)}x\right)^{-\frac{1}{2}}$ is chosen to be the line
$\arg(t-\frac{\tau }{2(\tau +1)}x)=\pi$.
\end{theorem}
We will present two different proofs of this in the paper. The first
one is a geometric proof which involves choosing a suitable set of
coordinates on $O(N)$ and decompose the Haar measure into two parts
so that the integral in (\ref{eq:jpdf}) can be evaluated. This will
be achieved in Sections \ref{se:haar} and \ref{se:int}. The second
proof is an algebraic proof that uses the Zonal polynomial expansion
to verify the formula in Theorem \ref{thm:main1}. This proof will be
given in the Appendix where integral formulae of the form
(\ref{eq:main1}) for the complex and quarternionic Wishart ensembles
will also be derived.
\begin{remark}
The integral formula derived here is very similar to a more general
formula in \cite{BE}, in which the matrix integral over $O(N)$ is
given by
\begin{equation*}
\int_{O(N)}e^{-\tr\left(XgYg^{-1}\right)}g^T\D g\propto
\int\frac{e^{\tr(S)}}{\prod_{j=1}^{N}\det(S-y_jX)}\D S
\end{equation*}
where the integral of $S$ is over $\sqrt{-1}$ times the space of
$N\times N$ real symmetric matrices and $y_j$ are the eigenvalues of
$Y$. The measure $\D S$ is the flat Lebesgue measure on this space.
\end{remark}
From the expression of the j.p.d.f., we see that the largest
eigenvalue distribution is given by
\begin{equation}\label{eq:Pmax}
\begin{split}
&\mathbb{P}(\lambda_{max}<z)=\int_{\lambda_1\leq\ldots\leq\lambda_N\leq z}\ldots\int P(\lambda)\D\lambda_1\ldots\D\lambda_N,\\
&=\tilde{Z}_{M,N}^{-1}
\int_{\Gamma}e^{Mt}\int_{\lambda_1\leq\ldots\leq\lambda_N\leq
z}\ldots\int|\Delta(\lambda)|\prod_{j=1}^Nw(\lambda_j)\D\lambda_1\ldots\D\lambda_Ndt
\end{split}
\end{equation}
where $w(x)$ is
\begin{equation}\label{eq:w1}
w(x)=e^{-\frac{M}{2}x}x^{\frac{M-N-1}{2}}\left(t-\frac{\tau
}{2(1+\tau )}x\right)^{-\frac{1}{2}}
\end{equation}
and $\Gamma$ is chosen such that it intersects $(0,\infty)$ encloses
the interval $[0,z]$.

We can analyze the integrand as in \cite{TW}, \cite{TW1} and
\cite{TW2}. By an identity of Brujin \cite{Brujin}, we can express
the multiple integral as a Pfaffian.
\begin{equation}\label{eq:pff}
\begin{split}
&\int_{\lambda_1\leq\ldots\leq\lambda_N\leq
y}\ldots\int|\Delta(\lambda)|\prod_{j=1}^Nw(\lambda_j)\D\lambda_1\ldots\D\lambda_N\\
&=Pf\left(\left<(1-\chi_{[z,\infty)})r_j(x),(1-\chi_{[z,\infty)})r_k(y)\right>_1\right).
\end{split}
\end{equation}
where $r_j(x)$ is an arbitrary sequence of degree $j$ monic
polynomials and $\left<f,g\right>_1$ is the skew product
\begin{equation}\label{eq:skewinner}
\left<f,g\right>_1=\int_0^{\infty}\int_0^{\infty}\epsilon(x-y)f(x)g(y)w(x)w(y)\D
x\D y.
\end{equation}
where $\epsilon(x)=\frac{1}{2}\mathrm{sgn}(x)$. In defining the skew
product, the contour of integration will be defined such that if $t$
is too close to $(0,\infty)$, then the interval $(0,\infty)$ will be
deformed appropriately into the upper or lower half plane such that
the integral is well defined. Such deformation will not affect the
value of the Pfaffian as $\Gamma$ will not intersect the integration
paths on the left hand side of (\ref{eq:pff}). Then by following the
method in \cite{TW}, \cite{TW1} and \cite{TW2}, we can write the
Pffafian as the square root of a Fredholm determinant. Let
$\mathcal{M}$ be the moment matrix with entries
$\left<r_j,r_k\right>_1$, then we have
\begin{equation*}
Pf\left(\left<(1-\chi_{[z,\infty)})(x)r_j(x),(1-\chi_{[z,\infty)})(y)r_k(y)\right>_1\right)
=\sqrt{\det\mathcal{M}(t)}\sqrt{\det\left(I-K\chi_{[z,\infty)}\right)},
\end{equation*}
where $K$ is the operator whose kernel is given by
\begin{equation}
K(x,y)=\begin{pmatrix}S_1(x,y)&-\frac{\p}{\p y}S_1(x,y)\\
IS_1(x,y)&S_1(y,x)\end{pmatrix}
\end{equation}
and $S_1(x,y)$ and $IS_1(x,y)$ are the kernels
\begin{equation}\label{eq:ker}
\begin{split}
S_1(x,y)&=-\sum_{j,k=0}^{N-1}r_j(x)w(x)\mu_{jk}\epsilon(r_kw)(y),\\
IS_1(x,y)&=-\sum_{j,k=0}^{N-1}\epsilon(r_jw)(x)\mu_{jk}\epsilon(r_kw)(y)
\end{split}
\end{equation}
and $\mu_{jk}$ is the inverse of the matrix $\mathcal{M}$. As shown
in \cite{W}, the kernel can now be expressed in terms of the
Christoffel Darboux kernel of some suitable orthogonal polynomials,
together with a correction term which gives rise to a finite rank
perturbation to the Christoffel Darboux kernel. In this paper, we
introduce a new proof of this using skew orthogonal polynomials and
their representations as multi-orthogonal polynomials. By using
ideas from \cite{AF} to write skew orthogonal polynomials in terms
of orthogonal polynomials, we can express the skew orthogonal
polynomials with respect to the weight $w(x)$ in terms of a sum of
Laguerre polynomials. Let $\pi_{k,1}$ be the monic skew orthogonal
polynomials with respect to the weight $w(x)$.
\begin{equation}\label{eq:sop}
\left<\pi_{2k+1,1},y^j\right>_1=\left<\pi_{2k,1},y^j\right>_1=0,\quad
j=0,\ldots,2k-1.
\end{equation}
Then we can write these down in terms of Laguerre polynomials.
\begin{proposition}
Let $L_k$ be the monic Laguerre polynomials respect to the weight
$w_0(x)$
\begin{equation*}
\int_0^{\infty}L_k(x)L_j(x)w_0(x)\D x=\delta_{jk}h_{j,0},\quad
w_0(x)=x^{M-N}e^{-Mx}.
\end{equation*}
If $\left<L_{2k-1},L_{2k-2}\right>_1\neq 0$, then the skew
orthogonal polynomials $\pi_{2k,1}$ and $\pi_{2k+1,1}$ both exist
and $\pi_{2k,1}$ is unique while $\pi_{2k+1,1}$ is unique up to an
addition of a multiple of $\pi_{2k,1}$. Moreover, we have
$\left<L_{2k},L_{2k-1}\right>_1=0$ and the skew orthogonal
polynomials are given by
\begin{equation*}
\begin{split}
\pi_{2k,1}&=L_{2k}-\frac{\left<L_{2k},L_{2k-2}\right>_1}{\left<L_{2k-1},L_{2k-2}\right>_1}L_{2k-1},\\
\pi_{2k+1,1}&=L_{2k+1}-\frac{\left<L_{2k+1},L_{2k-2}\right>_1}{\left<L_{2k-1},L_{2k-2}\right>_1}L_{2k-1}
+\frac{\left<L_{2k+1},L_{2k-1}\right>_1}{\left<L_{2k-1},L_{2k-2}\right>_1}L_{2k-2}+c\pi_{2k,1},
\end{split}
\end{equation*}
where $c$ is an arbitrary constant.
\end{proposition}
Next, by representing skew orthogonal polynomials as
multi-orthogonal polynomials and write them in terms of the solution
of a Riemann-Hilbert problem as in \cite{Pierce}, we can apply the
results of \cite{KDaem} and \cite{baikext} to express the kernel
$S_1(x,y)$ as a finite rank perturbation of the Christoffel Darboux
kernel of the Laguerre polynomials.
\begin{theorem}\label{thm:baik}
Let $S_1(x,y)$ defined by (\ref{eq:ker}) and choose the sequence of
monic polynomials $r_j(x)$ such that $r_j(x)$ are arbitrary degree
$j$ monic polynomials that are independent on $t$ and
$r_j(x)=\pi_{j,1}(x)$ for $j=N-2,N-1$. Then we have
\begin{equation}\label{eq:thmbaik}
\begin{split}
&S_1(x,y)-K_2(x,y)=\\
&\epsilon\left(\pi_{N+1,1}w\quad\pi_{N,1}w\right)(y)\begin{pmatrix}0&-\frac{M\tilde{\tau}}{2h_{N-1,0}}\\
-\frac{M\tilde{\tau}}{2h_{N-2,0}}&\frac{Mt-\tilde{\tau}(M+N)}{2h_{N-1,0}}\end{pmatrix}\begin{pmatrix}L_{2N-2}(x)\\
L_{2N-1}(x)\end{pmatrix}w(x)
\end{split}
\end{equation}
where $K_2(x,y)$ is the kernel of the Laguerre polynomials
\begin{equation*}
K_2(x,y)=\left(\frac{y(t-\tilde{\tau}y)}{x(t-\tilde{\tau}x)}\right)^{\frac{1}{2}}
w_0^{\frac{1}{2}}(x)w_0^{\frac{1}{2}}(y)\frac{L_{N}(x)L_{N-1}(y)-L_N(y)L_{N-1}(x)}{h_{N-1,0}(x-y)}
\end{equation*}
\end{theorem}
Note that the correction term on the right hand side of
(\ref{eq:thmbaik}) is the kernel of a finite rank operator. Its
asymptotics can be computed using the known asymptotics of the
Laguerre polynomials and the method in \cite{DG} and \cite{DGKV}.
The actual asymptotic analysis of this correction term, however, is
particularly tedious as one would need to compute the asymptotics of
the skew orthogonal polynomials up to the third leading order term
due to cancelations. To compute the contribution from the
determinant $\det\mathcal{M}$, we derive the following expression
for the logarithmic derivative of $\det\mathcal{M}$.
\begin{proposition}
Let $\mathcal{M}$ be the moment matrix with entries
$\left<r_j,r_k\right>_1$, where the sequence of monic polynomials
$r_j(x)$ is chosen such that $r_j(x)$ are arbitrary degree $j$ monic
polynomials that are independent on $t$ and $r_j(x)=\pi_{j,1}(x)$
for $j=N-2,N-1$. Then the logarithmic derivative of
$\det\mathcal{M}$ with respect to $t$ is given by
\begin{equation}
\frac{\p}{\p
t}\log\det\mathcal{M}=\int_{\mathbb{R}_+}\frac{S_1(x,x)}{t-\tilde{\tau}x}\D
x,
\end{equation}
\end{proposition}
This then allows us to express the largest eigenvalue distribution
$\mathbb{P}(\lambda_{max}<z)$ as an integral of Fredholm
determinant.
\begin{theorem}\label{thm:main2}
The largest eigenvalue distribution of the rank 1 real Wishart
ensemble can be written in the following integral form.
\begin{equation}\label{eq:Pmaxdet}
\begin{split}
\mathbb{P}(\lambda_{max}<z)=
C\int_{\Gamma}\exp\left(Mt+\int_{c_0}^t\int_{\mathbb{R}_+}\frac{S_1(x,x)}{s-\tilde{\tau}x}\D
x\D s\right)\sqrt{\det\left(I-K\chi_{[z,\infty)}\right)}\D t.
\end{split}
\end{equation}
for some constant $c_0$ and $K$ is the operator with kernel given by
(\ref{eq:ker}). The integration contour $\Gamma$ is a close contour
that encloses the interval $[0,z]$ in the anti-clockwise direction.
\end{theorem}
In the asymptotic limit, we will be able to evaluate the $t$
integral in (\ref{eq:Pmaxdet}) using steepest descent analysis. We
shall see that the phase transition occurs when the saddle point in
$t$ is such that the singularity $t/\tilde{\tau}$ in the weight
$w(x)$ lies within a distance of order $N^{-\frac{2}{3}}$ to the end
point $b_+$ in (\ref{eq:MP}). In this case, the factor
$(t-\tilde{\tau}x)^{-\frac{1}{2}}$ in the weight $w(x)$ will
significantly alter the behavior of the correction term in
(\ref{eq:thmbaik}) and gives us a phase transition in the largest
eigenvalue distribution.

In this first part of the paper, we shall carry out the analysis
when $N$ and $M$ are finite and establish the results that are
needed in the asymptotic analysis. Throughout the paper, we shall
assume that $N$ is even and that $M-N>0$.

\section{Haar measure on $SO(N)$}\label{se:haar}
In this section, we will find a convenient set of coordinate on
$O(N)$ to evaluate the integral
\begin{equation*}
\int_{O(N)} e^{-\frac{M}{2}\tr(\Sigma^{-1}gS
g^{-1})}g^Tdg
\end{equation*}
that appears in the expression of the j.p.d.f. (\ref{eq:jpdf}). As
both $\Sigma^{-1}$ and $S$ are symmetric matrices, they can be
diagonalized by matrices in $O(N)$. We can therefore replace both
$\Sigma^{-1}$ and $S$ by the diagonal matrices $\Sigma_d^{-1}$ and
$\Lambda_d$.
\begin{equation*}
\begin{split}
\Sigma_d^{-1}&=\diag\left(\frac{1}{1+\tau_1},\ldots,\frac{1}{1+\tau _N}\right),\\
\Lambda_d&=\diag\left(\lambda_1,\ldots,\lambda_N\right)
\end{split}
\end{equation*}
The group $O(N)$ has two connected components, $SO(N)$ and $O_-(N)$
that consists of orthogonal matrices that have determinant $1$ and
$-1$ respectively. Let $T$ be the matrix
\begin{equation*}
T=\begin{pmatrix}0&1&0\\
1&0&0\\
0&0&I_{N-2}\end{pmatrix},
\end{equation*}
then the left multiplication by $T$ defines an diffeomorphism from
$O_-(N)$ to $SO(N)$. In particular, we can write the integral over
$O(N)$ in (\ref{eq:jpdf}) as
\begin{equation*}
\begin{split}
I(\Sigma,\Lambda)&=\int_{O(N)}e^{-\frac{M}{2}\tr(\Sigma_d^{-1}g\Lambda_d
g^{-1})}g^Tdg,\\&=
\int_{SO(N)}e^{-\frac{M}{2}\tr(\Sigma_d^{-1}g\Lambda_d
g^{-1})}g^Tdg+\int_{O_-(N)}e^{-\frac{M}{2}\tr(\Sigma_d^{-1}g\Lambda_d
g^{-1})}g^Tdg\\
&=\int_{SO(N)}e^{-\frac{M}{2}\tr(\Sigma_d^{-1}g\Lambda_d
g^{-1})}g^Tdg+\int_{SO(N)}e^{-\frac{M}{2}\tr(\Sigma_d^{-1}Tg\Lambda_d
g^{-1}T^{-1})}g^Tdg\\
&=\int_{SO(N)}e^{-\frac{M}{2}\tr(\Sigma_d^{-1}g\Lambda_d
g^{-1})}g^Tdg+\int_{SO(N)}e^{-\frac{M}{2}\tr(\tilde{\Sigma}_d^{-1}g\Lambda_d
g^{-1})}g^Tdg,
\end{split}
\end{equation*}
where $\tilde{\Sigma}_d$ is the diagonal matrix with the first two
entries of $\Sigma_d$ swapped.
\begin{equation*}
\tilde{\Sigma}_d^{-1}=\diag\left(\frac{1}{1+\tau _2},\frac{1}{1+\tau
_1}\ldots,\frac{1}{1+\tau _N}\right).
\end{equation*}
Note that $g^Tdg$ is also the Haar measure on $SO(N)$.

As we are considering the rank 1 spiked model, we let $\tau
_1=\ldots=\tau _{N-1}=0$ and $\tau _N=\tau $. Therefore
$\tilde{\Sigma}_d=\Sigma_d$ and we have
\begin{equation}\label{eq:I}
\begin{split}
I(\Sigma,\Lambda)
&=2\int_{SO(N)}e^{-\frac{M}{2}\tr(\Sigma_d^{-1}g\Lambda_d
g^{-1})}g^Tdg
\end{split}
\end{equation}
Let $g_{ij}$ be the entries of $g\in SO(N)$. Then the integral $I$
can be written as
\begin{equation*}
\begin{split}
I(\Sigma,\Lambda)&=2\int_{SO(N)}e^{-\frac{M}{2}\tr(\Sigma_d^{-1}g\Lambda_d
g^{-1})}g^Tdg,\\
&=2\int_{SO(N)}e^{-\frac{M}{2}\tr(\left(\Sigma_d^{-1}-I_N\right)g\Lambda_d
g^{-1})}e^{-\frac{M}{2}\tr(g\Lambda_d g^{-1})}g^Tdg,\\
&=2\prod_{j=1}^Ne^{-\frac{M}{2}\lambda_j}\int_{SO(N)}e^{\frac{\tau M}{2(1+\tau )}\sum_{j=1}^N\lambda_jg_{jN}^2}g^Tdg,\\
\end{split}
\end{equation*}
We will now find an expression of the Haar measure that allows us to
compute the integral $I(\Sigma,\Lambda)$.

First we will define a set of coordinates on $SO(N)$ that is
convenient for our purpose. We will then express the Haar measure on
$SO(N)$ in terms of these coordinates.

An element $g\in SO(n)$ can be written in the following form
\begin{equation*}
g=\left(\vec{g}_1,\ldots,\vec{g}_n\right),\quad |\vec{g}_i|=1,\quad
\vec{g}_i\cdot\vec{g}_j=\delta_{ij},\quad i,j=1,\ldots, n.
\end{equation*}
This represents $SO(N)$ as the set of orthonormal frames in
$\mathbb{R}^N$ with positive orientation whose coordinate axis are
given by the vectors $\vec{g}_i$. As the vector $\vec{g}_N$ is a
unit vector, we can write its components as
\begin{equation}\label{eq:angles}
\begin{split}
g_{1N}&=\cos\phi_1,\quad
g_{jN}=\prod_{k=1}^{j-1}\sin\phi_k\cos\phi_j,\quad
j=2,\ldots,n-1,\\
g_{NN}&=\prod_{k=1}^{N-1}\sin\phi_k
\end{split}
\end{equation}
The remaining vectors $\vec{g}_1,\ldots,\vec{g}_{N-1}$ form an
orthonormal frame with positive orientation in a copy of
$\mathbb{R}^{N-1}$ that is orthogonal to $\vec{g}_N$. Therefore the
set of vectors $\vec{g}_1,\ldots,\vec{g}_{N-1}$ can be identified
with $SO(N-1)$. To be precise, let $\vec{u}$ be a unit vector in
$\mathbb{R}^N$ and let $G(\vec{u})\in SO(N)$ be a matrix that maps
$\vec{u}$ to the vector $\left(0,\ldots,0,1\right)^T$. Then since
$G$ is orthogonal, we have
\begin{equation}\label{eq:Gvecn}
G(\vec{g}_N)\vec{g}_j=\left(v_{j1},\ldots,v_{j,N-1},0\right)^T,\quad
j<N
\end{equation}
In particular, the matrix $V$ whose entries are given by $v_{ij}$
for $1\leq i,j\leq N-1$ is in $SO(N-1)$. A set of coordinates on
$SO(N)$ can therefore be given by
\begin{equation}\label{eq:coord}
g=\left(\vec{g}_N,V\right).
\end{equation}
In the above equation, $\vec{g}_N$ is identified with the
coordinates $\phi_{j}$ in (\ref{eq:angles}), while the matrix $V$
identify with the coordinates on $SO(N-1)$ that correspond to $V$.
In terms of these coordinates, the left action of an element $S\in
SO(N)$ on $g$ is given by the following.
\begin{equation*}
\begin{split}
Sg&=\left(S\vec{g}_1,\ldots,S\vec{g}_{N-1},S\vec{g}_N\right)^T \\
&=\left(SG(\vec{g}_N)^{-1}\vec{v}_1,\ldots,SG(\vec{g}_N)^{-1}\vec{v}_{N-1},S\vec{g}_N\right)^T
\end{split}
\end{equation*}
Then as in (\ref{eq:Gvecn}), we have
\begin{equation*}
G(S\vec{g}_N)SG(\vec{g}_N)^{-1}\vec{v}_j=\left(\tilde{v}_{j1},\ldots,\tilde{v}_{j,N-1},0\right)^T.
\end{equation*}
The matrix $\tilde{V}$ whose entries are given by $\tilde{v}_{ij}$
are again in $SO(N-1)$, therefore the matrix
$G(S\vec{g}_N)SG(\vec{g}_N)^{-1}$ is of the form
\begin{equation}\label{eq:action}
G(S\vec{g}_N)SG(\vec{g}_N)^{-1}=\begin{pmatrix}\tilde{S}_{N-1}&\vec{s}\\
                                               0&s_N\end{pmatrix}
\end{equation}
From the fact that $G(S\vec{g}_N)SG(\vec{g}_N)^{-1}$ is an
orthogonal matrix, it is easy to check that $\vec{s}=0$ and $s_N=\pm
1$. To determine $s_N$, let us consider the action of
$G(S\vec{g}_N)SG(\vec{g}_N)^{-1}$ on $(0,0,\ldots,1)^T$. We have
\begin{equation*}
G(S\vec{g}_N)SG(\vec{g}_N)^{-1}(0,0,\ldots,1)^T=G(S\vec{g}_N)S\vec{g}_N=(0,0,\ldots,1)^T
\end{equation*}
Therefore $s_N=1$ and $\tilde{S}_{N-1}$ is in $SO(N-1)$. The action
of $S$ on $g$ is therefore given by
\begin{equation}\label{eq:Saction}
Sg=\left(S\vec{g}_N,\tilde{S}_{N-1}V\right).
\end{equation}
We will now write the Haar measure on $SO(N)$ in terms the
coordinates (\ref{eq:coord}). These coordinates give a local
diffeomorphism between $SO(N)$ and $S^{N-1}\times SO(N-1)$ as
$\vec{g}_N\in S^{N-1}$ and $V\in SO(N-1)$. Let $dX$ be a measure on
$S^{N-1}$ that is invariant under the action of $SO(N)$ and $V^TdV$
be the Haar measure on $SO(N-1)$, then the following measure
\begin{equation*}
dH=dX\wedge V^TdV,
\end{equation*}
is invariant under the left action of $SO(N)$. Let $S\in SO(N)$,
then its action on the point $(\vec{g}_N,V)$ is given by
(\ref{eq:Saction}), where $\tilde{S}_{N-1}$ depends only on the
coordinates $\phi_1,\ldots,\phi_{N-1}$. Therefore under the action
of $S$, the measure $dH$ becomes
\begin{equation}\label{eq:dehaar}
dH\rightarrow dX\wedge
V^T\tilde{S}_{N-1}^T\tilde{S}_{N-1}dV=dX\wedge V^TdV,
\end{equation}
as $dX$ is invariant under the action of $S$. Therefore if we can
find a measure on $S^{N-1}$ that is invariant under the action of
$SO(N)$, then $dX\wedge V^TdV$ will give us a left invariant measure
on $SO(N)$. Since the left invariant measure on a compact group is
also right invariant, this will give us the Haar measure on $SO(N)$.
As the metric on $S^{N-1}$ is invariant under the action of $SO(N)$,
it is clear that the volume form on $S^{N-1}$ is invariant under the
action of $SO(N)$. Let $dX$ be the volume form on $S^{N-1}$, then
from (\ref{eq:dehaar}), we see that the measure $dX\wedge V^TdV$ is
invariant under the action of $SO(N)$.
\begin{proposition}\label{pro:haar}
Let $dX$ be the volume form on $S^{N-1}$ given by
\begin{equation*}
\begin{split}
dX=\sin^{N-2}(\phi_1)\sin^{N-1}(\phi_2)\ldots\sin(\phi_{N-2})\wedge_{j=1}^{N-1}d\phi_j
\end{split}
\end{equation*}
in terms of the coordinates $\phi_1,\ldots,\phi_{N-1}$ in
(\ref{eq:angles}) and (\ref{eq:coord}), then the Haar measure on
$SO(N)$ is equal to a constant multiple of
\begin{equation*}
dH=dX\wedge V^TdV,
\end{equation*}
where $V^TdV$ is the Haar measure on $SO(N-1)$ in terms of the
coordinates (\ref{eq:coord}).
\end{proposition}
We can now compute the integral $I(\Sigma,\Lambda)$.

\section{An integral formula for the j.p.d.f.}\label{se:int}
By using the expression of the Haar measure derived in the last
section, we can now write the integral $I(\Sigma,\Lambda)$ as
\begin{equation*}
\begin{split}
I(\Sigma,\Lambda)&=2\prod_{j=1}^Ne^{-\frac{M}{2}\lambda_j}\int_{SO(N)}e^{\frac{\tau M}{2(1+\tau )}\sum_{j=1}^N\lambda_jg_{jN}^2}g^Tdg,\\
&=2\prod_{j=1}^Ne^{-\frac{M}{2}\lambda_j}\int_{SO(N-1)}V^TdV\int_{S^{N-1}}e^{\frac{\tau M}{2(1+\tau )}\sum_{j=1}^N\lambda_jg_{jN}^2}dX,\\
&=2C\prod_{j=1}^Ne^{-\frac{M}{2}\lambda_j}\int_{S^{N-1}}e^{\frac{\tau
M}{2(1+\tau )}\sum_{j=1}^N\lambda_jg_{jN}^2}dX,
\end{split}
\end{equation*}
for some constant $C$, where the $N-1$ sphere $S^{N-1}$ in the above
formula is defined by $\sum_{j=1}^Ng_{jN}^2=1$ and $dX$ is the
volume form on it. If we let $g_{jN}=x_j$, then the above can be
written as
\begin{equation}\label{eq:Iint}
\begin{split}
I(\Sigma,\Lambda)&=2C\prod_{j=1}^Ne^{-\frac{M}{2}\lambda_j}\int_{\mathbb{R}^N}e^{\frac{\tau
M}{2(1+\tau )}\sum_{j=1}^N\lambda_jx_{j}^2}
\delta\left(\sum_{j=1}^Nx_j^2-1\right)dx_1\ldots dx_N.
\end{split}
\end{equation}
This can be seen most easily by the use of polar coordinates in
$\mathbb{R}^N$, which are given by
\begin{equation*}
\begin{split}
x_{1}&=r\cos\phi_1,\quad
x_{j}=r\prod_{k=1}^{j-1}\sin\phi_k\cos\phi_j,\quad
j=2,\ldots,N-1,\\
x_{N}&=r\prod_{k=1}^{N-1}\sin\phi_k,
\end{split}
\end{equation*}
Then the volume form in $\mathbb{R}^N$ is given by
\begin{equation*}
dx_1\ldots
dx_N=r^{N-1}\sin^{N-2}\phi_1\ldots\sin\phi_{N-2}drd\phi_1\ldots
d\phi_{N-1}
\end{equation*}
Therefore in terms of polar coordinates, we have
\begin{equation*}
\begin{split}
&\int_{\mathbb{R}^N}e^{\frac{\tau M}{2(1+\tau
)}\sum_{j=1}^N\lambda_jx_{j}^2}
\delta\left(\sum_{j=1}^Nx_j^2-1\right)dx_1\ldots dx_N\\
&=\int_{0}^{\pi}d\phi_1\int_0^{2\pi}d\phi_2\ldots\int_0^{2\pi}d\phi_{N-1}
\int_{0}^{\infty}dr\delta\left(\sum_{j=1}^Nr^2-1\right)r^{N-1}\\
&\times e^{\frac{\tau M}{2(1+\tau )}\sum_{j=1}^N\lambda_jx_{j}^2}
\sin^{N-2}\phi_1\ldots\sin\phi_{N-2}\\
&=\int_{S^{N-1}}e^{\frac{\tau M}{2(1+\tau
)}\sum_{j=1}^N\lambda_jx_{j}^2} dX.
\end{split}
\end{equation*}
To compute the integral $I(\Sigma,\Lambda)$, we use a method in the
studies of random pure quantum systems \cite{Maj}. The idea is to
consider the Laplace transform of the function $I(\Sigma,\Lambda,t)$
defined by
\begin{equation*}
I(\Sigma,\Lambda,t)=2C\prod_{j=1}^Ne^{-\frac{M}{2}\lambda_j}\int_{\mathbb{R}^N}e^{\frac{\tau
M}{2(1+\tau )}\sum_{j=1}^N\lambda_jx_{j}^2}
\delta\left(\sum_{j=1}^Nx_j^2-t\right)dx_1\ldots dx_N,
\end{equation*}
then $I(\Sigma,\Lambda,1)=I(\Sigma,\Lambda)$. The Laplace transform
of $I(\Sigma,\Lambda,t)$ in the variable $t$ is given by
\begin{equation*}
\int_0^{\infty}e^{-st}I(\Sigma,\Lambda,t)dt=
2C\prod_{j=1}^Ne^{-\frac{M}{2}\lambda_j}\int_{\mathbb{R}^N}e^{\sum_{j=1}^N\left(-s+\frac{\tau
M}{2(1+\tau )}\lambda_j\right)x_{j}^2} dx_1\ldots dx_N
\end{equation*}
Then, provided $\mathrm{Re}(s)>\max_j\left(\lambda_j\right)$, the
integral can be computed explicitly to obtain
\begin{equation*}
\int_0^{\infty}e^{-st}I(\Sigma,\Lambda,t)dt=
2C\prod_{j=1}^Ne^{-\frac{M}{2}\lambda_j}\left(s-\frac{\tau
M}{2(1+\tau )}\lambda_j\right)^{-\frac{1}{2}}
\end{equation*}
Taking the inverse Laplace transform, we obtain an integral
expression for $I(\Sigma,\Lambda)$.
\begin{equation*}
I(\Sigma,\Lambda)= \frac{C}{\pi i}
\int_{\Gamma}e^s\prod_{j=1}^Ne^{-\frac{M}{2}\lambda_j}\left(s-\frac{\tau
M}{2(1+\tau )}\lambda_j\right)^{-\frac{1}{2}}ds,
\end{equation*}
where $\Gamma$ is a contour that encloses all the points $\frac{\tau
M}{2(1+\tau )}\lambda_1,\ldots,\frac{\tau M}{2(1+\tau )}\lambda_N$
that is oriented in the counter-clockwise direction. Rescaling the
variable $s$ to $s=Mt$, we obtain
\begin{equation*}
I(\Sigma,\Lambda)= \frac{M^{1-\frac{N}{2}}C}{\pi i}
\int_{\Gamma}e^{Mt}\prod_{j=1}^Ne^{-\frac{M}{2}\lambda_j}\left(t-\frac{\tau
}{2(1+\tau )}\lambda_j\right)^{-\frac{1}{2}}dt,
\end{equation*}
This then give us an integral expression for the j.p.d.f.
\begin{theorem}\label{thm:jpdf}
Let the non-trivial eigenvalue in the covariance matrix $\Sigma$ be
$1+\tau $. Then the j.p.d.f. of the eigenvalues in the rank 1 real
Wishart spiked model with covariance matrix $\Sigma$ is given by
\begin{equation}\label{eq:dist}
P(\lambda)=\tilde{Z}_{M,N}^{-1}
\int_{\Gamma}|\Delta(\lambda)|e^{Mt}\prod_{j=1}^Ne^{-\frac{M}{2}\lambda_j}\lambda_j^{\frac{M-N-1}{2}}\left(t-\frac{\tau
}{2(1+\tau )}\lambda_j\right)^{-\frac{1}{2}}dt,
\end{equation}
where $\Gamma$ is a contour that encloses all the points $\frac{\tau
}{2(1+\tau )}\lambda_1,\ldots,\frac{\tau }{2(1+\tau )}\lambda_N$
that is oriented in the counter-clockwise direction and
$\tilde{Z}_{M,N}$ is the normalization constant. The branch cuts of
the square root $\left(a-\frac{\tau }{2(\tau
+1)}x\right)^{-\frac{1}{2}}$ are chosen to be the line
$\arg(a-\frac{\tau }{2(\tau +1)}x)=\pi$.
\end{theorem}
For the purpose of computing the largest eigenvalue distribution
$\mathbb{P}\left(\lambda_{max}\leq z\right)$, we can assume that the
eigenvalues are all smaller than or equal to a constant $z$.
\section{Skew orthogonal polynomials}
As explain in the introduction, we need to find the skew orthogonal
polynomials with the weight (\ref{eq:w1}). Let us denote
$\frac{\tau}{2(\tau+1)}$ by $\tilde{\tau}$ and consider the skew
orthogonal polynomials with respect to the weight
\begin{equation}\label{eq:weight}
\begin{split}
 w(x)&=e^{-\frac{Mx}{2}}x^{\frac{M-N-1}{2}}(t-\tilde{\tau
}x)^{-\frac{1}{2}}.
\end{split}
\end{equation}
We shall use the ideas in \cite{AF} to express the skew orthogonal
polynomials in terms of a linear combinations of Laguerre
polynomials.

Let $H_j(x)$ to be the degree $j+2$ polynomial
\begin{equation*}
H_j(x)=\frac{d}{dx}\left(x^{j+1}(t-\tilde{\tau}x)w(x)\right)w^{-1}(x),\quad
j\geq 0.
\end{equation*}
Then as we assume $M-N>0$, it is easy to see that
\begin{equation}\label{eq:parts}
\left<f(x),H_j(y)\right>_1=\left<f(x),x^j\right>_2,
\end{equation}
for any $f(x)$ such that $\int_0^{\infty}f(x)w(x)\D x$ is finite,
where the product $\left<\right>_2$ is defined by
\begin{equation}\label{eq:w0}
\left<f(x)g(x)\right>_2=\int_{0}^{\infty}f(x)g(x)w_0(x)dx,\quad
w_0(x)=x^{M-N}e^{-Mx}.
\end{equation}
Note that $w_0(x)$ is not the square of $w(x)$. The fact that
$w_0(x)$ is the weight for the Laguerre polynomials allows us to
express the skew orthogonal polynomials for the weight
(\ref{eq:weight}) in terms of Laguerre polynomials.

In particular, this implies that the conditions (\ref{eq:sop}) is
equivalent to the following conditions
\begin{equation}\label{eq:sop2}
\begin{split}
\left<\pi_{2k,1},y^j\right>_1&=0,\quad j=0,1,\\
\left<\pi_{2k,1},y^j\right>_2&=0,\quad j=0,\ldots,2k-3.
\end{split}
\end{equation}
and the exactly same conditions for $\pi_{2k+1,1}(x)$. In
particular, the second condition implies the skew orthogonal
polynomials can be written as
\begin{equation*}
\begin{split}
\pi_{2k,1}(x)&=L_{2k}(x)+\gamma_{2k,1}L_{2k-1}(x)+\gamma_{2k,2}L_{2k-2}(x),\\
\pi_{2k+1,1}(x)&=L_{2k+1}(x)+\gamma_{2k+1,0}L_{2k}(x)+\gamma_{2k+1,1}L_{2k-1}(x)+\gamma_{2k+1,2}L_{2k-2}(x),
\end{split}
\end{equation*}
where $L_j(x)$ are the degree $j$ monic Laguerre polynomials that
are orthogonal with respect to the weight $w_0(x)$.
\begin{equation}\label{eq:laguerre}
\begin{split}
L_{n}(x)&=\frac{(-1)^ne^{Mx}x^{-M+N}}{M^n}\frac{d^n}{dx^n}\left(e^{-Mx}x^{n+M-N}\right),\\
&=x^n-\frac{(M-N+n)n}{M}x^{n-1}+O(x^{n-2}).
\end{split}
\end{equation}
The constants $\gamma_{k,j}$ are to be determined from the first
condition in (\ref{eq:sop2}). We will now show that if
$\left<L_{2k-1},L_{2k-2}\right>_1\neq 0$, then the skew orthogonal
polynomials $\pi_{2k,1}$ and $\pi_{2k+1,1}$ exist and that
$\pi_{2k,1}$ is unique. First let us show that the first condition
in (\ref{eq:sop2}) is equivalent to
\begin{equation*}
\begin{split}
\left<\pi_{2k,1},L_{2k-j}\right>_1&=0,\quad j=1,2.
\end{split}
\end{equation*}
To do this, we will first define a map $\varrho_N$ from the span of
$L_{2k-1}$ and $L_{2k-2}$ to the span of $y$ and $1$.

Let $P(x)$ be a polynomial of degree $m$. Then we can write the
polynomial $P(x)$ as
\begin{equation}\label{eq:PRq}
P(x)=\frac{d}{dx}\left(q(x)x(t-\tilde{\tau}
x)w(x)\right)w^{-1}(x)+R(x)
\end{equation}
where $q(x)$ is a polynomial of degree $m-2$ and $R(x)$ is a
polynomial of degree less than or equal to 1. By writing down the
system of linear equations satisfied by the coefficients of $q(x)$
and $R(x)$, we see that the polynomials $q(x)$ and $R(x)$ are
uniquely defined for any given $P(x)$. In particular, the map
$f:P(x)\mapsto R(x)$ is a well-defined linear map from the space of
polynomial to the space of polynomials of degrees less than or equal
to $1$. Let $\varrho_{k}$ be the following restriction of this map.
\begin{definition}\label{de:fn}
For any polynomial $P(x)$, let $f$ be the map that maps $P(x)$ to
$R(x)$ in (\ref{eq:PRq}). Then the map $\varrho_{k}$ is the
restriction of $f$ to the linear subspace spanned by the orthogonal
polynomials $L_{k},L_{k-1}$.
\end{definition}
We then have the following.
\begin{lemma}\label{le:linear}
If $\left<L_{k},L_{k-1}\right>_1\neq 0$, then the map $\varrho_{k}$
is invertible.
\end{lemma}
\begin{proof} Suppose there is exists non-zero constants $a_1$ and
$a_2$ such that
\begin{equation*}
a_1L_{k}+a_2L_{k-1}=\frac{d}{dx}(q(x)x(t-\tilde{\tau} x)w)w^{-1}
\end{equation*}
for some polynomial $q(x)$ of degree $k-2$, then by taking the skew
product $\left<\right>_1$ of this polynomial with $L_{k}$, we obtain
\begin{equation*}
a_2\left<L_{k-1},L_{k}\right>_1=
\left<a_1L_{k}+a_2L_{k-1},L_{k}\right>_1=\left<q(x),L_{k}\right>_2=0.
\end{equation*}
As $q(x)$ is of degree $k-2$. Since
$\left<L_{k-1},L_{k}\right>_1\neq 0$, this shows that $a_2=0$. By
taking the skew product with $L_{k-1}$, we conclude that $a_1=0$ and
hence the map $\varrho_{k}$ has a trivial kernel.
\end{proof}
In particular, we have the following.
\begin{corollary}\label{cor:linear}
If $k$ is even, then $\left<L_{k},L_{k-1}\right>_1=0$.
\end{corollary}
\begin{proof}
Let $q(x)$ be a polynomial of degree $k-2$ that satisfies the
following conditions
\begin{equation}\label{eq:qcon}
\int_{\mathbb{R}_+}\frac{d}{dx}(q(x)w_4(x))x^jw_4(x)\D x=0, \quad
j=0,\ldots, k-2,
\end{equation}
where
$w_4(x)=x^{\frac{M-N+1}{2}}(t-\tilde{\tau}x)^{\frac{1}{2}}e^{-\frac{Mx}{2}}$.
A non trivial polynomial $q(x)$ of degree $k-2$ that satisfies these
conditions exists if and only if the moment matrix with entries
$\int_{\mathbb{R}_+}\frac{d}{dx}(x^iw_4(x))x^jw_4(x)\D x$ has a
vanishing determinant. For even $k$, the moment matrix is of odd
dimension and anti-symmetric and hence its determinant is always
zero.

Assuming $k$ is even and let $q(x)$ be a polynomial that satisfies
(\ref{eq:qcon}). By taking the inner product $\left<\right>_2$ with
$x^j$, we see that there exists non-zero constants $a_1$ and $a_2$
such that
\begin{equation*}
a_1L_{k}+a_2L_{k-1}=\frac{d}{dx}(q(x)x(t-\tilde{\tau} x)w)w^{-1},
\end{equation*}
Therefore by Lemma \ref{le:linear}, we see that if $k$ is even, we
will have $\left<L_{k},L_{k-1}\right>_1=0$.
\end{proof}
Lemma \ref{le:linear} shows that if
$\left<L_{i},L_{i-1}\right>_1\neq 0$, then there exists two
independent polynomials $R_0(y)$ and $R_1(y)$ in the span of $y$ and
$1$ such that $R_j(y)=\varrho_{i}(L_{i-j})$. Then we have
\begin{equation*}
R_j(y)=-\frac{d}{dy}(q_j(y)y(t-\tilde{\tau}y)w)w^{-1}+L_{i-j}(y),\quad
j=0,1.
\end{equation*}
In particular, the skew product of $R_j(y)$ with $L_{i-l}$, $l<2$ is
given by
\begin{equation*}
\left<L_{i-l}(x),R_j(y)\right>_1=-\left<L_{i-l}q_j\right>_2+\left<L_{i-l},L_{i-j}\right>_1.
\end{equation*}
As $q_j$ is a polynomial of degree less than or equal to $i-2$ and
$l<2$, the first term on the right hand side is zero. Therefore we
have
\begin{equation}\label{eq:prod}
\left<L_{i-l}(x),R_j(y)\right>_1=\left<L_{i-l},L_{i-j}\right>_1,\quad
l< 2,\quad j=0,1.
\end{equation}
We can now show that the skew orthogonal polynomials $\pi_{2k,1}$
and $\pi_{2k+1,1}$ exist if $\left<L_{2k-1},L_{2k-2}\right>_1\neq
0$.
\begin{proposition}\label{pro:p2N}
If $\left<L_{2k-1},L_{2k-2}\right>_1\neq 0$, then the skew
orthogonal polynomials $\pi_{2k,1}$ and $\pi_{2k+1,1}$ both exist
and $\pi_{2k,1}$ is unique while $\pi_{2k+1,1}$ is unique up to an
addition of a multiple of $\pi_{2k,1}$. Moreover, we have
$\left<L_{2k},L_{2k-1}\right>_1=0$ and the skew orthogonal
polynomials are given by
\begin{equation}\label{eq:p2N}
\begin{split}
\pi_{2k,1}&=L_{2k}-\frac{\left<L_{2k},L_{2k-2}\right>_1}{\left<L_{2k-1},L_{2k-2}\right>_1}L_{2k-1},\\
\pi_{2k+1,1}&=L_{2k+1}-\frac{\left<L_{2k+1},L_{2k-2}\right>_1}{\left<L_{2k-1},L_{2k-2}\right>_1}L_{2k-1}
+\frac{\left<L_{2k+1},L_{2k-1}\right>_1}{\left<L_{2k-1},L_{2k-2}\right>_1}L_{2k-2}+c\pi_{2k,1},
\end{split}
\end{equation}
for $k\geq 2$, where $c$ is an arbitrary constant.
\end{proposition}
\begin{proof} Let $\pi_{2k,1}$ and $\pi_{2k+1,1}$ be polynomials
defined by
\begin{equation*}
\begin{split}
\pi_{2k,1}(x)&=L_{2k}(x)+\gamma_{2k,1}L_{2k-1}(x)+\gamma_{2k,2}L_{2k-2}(x),\\
\pi_{2k+1,1}(x)&=L_{2k+1}(x)+\gamma_{2k+1,1}L_{2k-1}(x)+\gamma_{2k+1,2}L_{2k-2}(x),
\end{split}
\end{equation*}
for some constants $\gamma_{j,k}$. If we can show that
$\left<\pi_{2k-l,1},y^j\right>_1=0$ for $j=0,1$ and $l=-1,0$, then
$\pi_{2k-l,1}$ will be the skew orthogonal polynomial. Let $R_0$ and
$R_1$ be the images of $L_{2k-1}$ and $L_{2k-2}$ under the map
$\varrho_{2k-1}$. Then by the assumption in the Proposition, they
are independent in the span of $y$ and $1$. Therefore the conditions
$\left<\pi_{2k-l,1},y^j\right>_1=0$ are equivalent to
$\left<\pi_{2k-l,1},R_j(y)\right>_1=0$. By taking $i=2k-1$ in
(\ref{eq:prod}), we see that this is equivalent to
$\left<\pi_{2k-l,1},L_{2k-1-j}\right>_1=0$. This implies
\begin{equation*}
\begin{split}
\left<\pi_{2k,1},L_{2k-1}\right>_1&=\left<L_{2k},L_{2k-1}\right>_1+\gamma_{2k,2}\left<L_{2k-2},L_{2k-1}\right>_1=0,\\
\left<\pi_{2k,1},L_{2k-2}\right>_1&=\left<L_{2k},L_{2k-2}\right>_1+\gamma_{2k,1}\left<L_{2k-1},L_{2k-2}\right>_1=0.
\end{split}
\end{equation*}
Hence we have
\begin{equation*}
\gamma_{2k,1}=-\frac{\left<L_{2k},L_{2k-2}\right>_1}{\left<L_{2k-1},L_{2k-2}\right>_1},\quad
\gamma_{2k,2}=\frac{\left<L_{2k},L_{2k-1}\right>_1}{\left<L_{2k-1},L_{2k-2}\right>_1},
\end{equation*}
which exist and are unique as $\left<L_{2k-1},L_{2k-2}\right>_1\neq
0$. This determines $\pi_{2k,1}$ uniquely. By Corollary
\ref{cor:linear}, we have $\left<L_{2k},L_{2k-1}\right>_1=0$ and
hence $\gamma_{2k,2}=0$. Similarly, the coefficients for
$\pi_{2k+1,1}$ are
\begin{equation*}
\gamma_{2k+1,1}=-\frac{\left<L_{2k+1},L_{2k-2}\right>_1}{\left<L_{2k-1},L_{2k-2}\right>_1},\quad
\gamma_{2k+1,2}=\frac{\left<L_{2k+1},L_{2k-1}\right>_1}{\left<L_{2k-1},L_{2k-2}\right>_1}.
\end{equation*}
Again, these coefficients exist and are unique. However, as
$\left<\pi_{2k,1},\pi_{2k,1}\right>_1=0$ and
$\left<\pi_{2k,1},y^j\right>_1=0$ for $j=0,\ldots,2k-1$, adding any
multiple of $\pi_{2k,1}$ to $\pi_{2k+1,1}$ will not change the
orthogonality conditions $\left<\pi_{2k+1,1},y^j\right>_1=0$ that is
satisfied by $\pi_{2k+1,1}$ and hence $\pi_{2k+1,1}$ is only
determined up to the addition of a multiple of $\pi_{2k,1}$.
\end{proof}
\section{The Christoffel Darboux formula for the kernel}
In \cite{Pierce}, skew orthogonal polynomials were interpreted as
multi-orthogonal polynomials and represented as the solution of a
Riemann-Hilbert problem. This representation allows us to us the
results in \cite{KDaem} to derive a Christoffel-Darboux formula for
the kernel (\ref{eq:ker}) in terms of the Riemann-Hilbert problem.

Let us recall the definitions of multi-orthogonal polynomials. First
let the weights $w_0$, $w_1$ and $w_2$ be
\begin{equation*}
w_0(x)=x^{M-N}e^{-Mx},\quad
w_l(x)=w(x)\int_{\mathbb{R}_+}\epsilon(x-y)L_{N-l-2}(y)w(y)\D
y,\quad l=1,2.
\end{equation*}
Note that the weights $w_l(x)$ are defined with the polynomials
$L_{N-3}$ and $L_{N-4}$ instead of $L_{N-1}$ and $L_{N-2}$. This is
because the construction below involves the polynomial $\pi_{N-2,1}$
as well as $\pi_{N,1}$. By taking $i=N-3$ in (\ref{eq:prod}), we see
that the orthogonality conditions for $\pi_{N,1}$ is also equivalent
to
\begin{equation*}
\begin{split}
\left<\pi_{N,1},x^j\right>_2=0,\quad j=0,\ldots,N-3,\\
\left<\pi_{N,1},L_{N-j}\right>_2=0,\quad j=3,4,
\end{split}
\end{equation*}
provided $\left<L_{N-3},L_{N-4}\right>_1$ is also non-zero.

Then the orthognoal polynomials of type II $P^{II}_{N,l}(x)$,
\cite{Ap1}, \cite{Ap2}, \cite{BK1}, \cite{vanGerKuij} are
polynomials of degree $N-1$ such that
\begin{equation}\label{eq:mop1}
\begin{split}
\int_{\mathbb{R}_+}P^{II}_{N,l}(x)x^jw_0(x)\D x&=0,\quad 0\leq j\leq N-3,\\
\int_{\mathbb{R}_+}P^{II}_{N,l}(x)w_m(x)\D x&=-2\pi
i\delta_{lm},\quad l,m=1,2.
\end{split}
\end{equation}
\begin{remark} More accurately, these are in fact the
multi-orthogonal polynomials with indices $(N-2,\vec{e}_l)$, where
$\vec{e}_1=(0,1)$ and $\vec{e}_2=(1,0)$.
\end{remark}
We will now define the multi-orthogonal polynomials of type I. Let
$P^{I}_{N,l}(x)$ be a function of the following form
\begin{equation}\label{eq:typeI}
P^{I}_{N,l}(x)=B_{N,l}(x)w_0+\sum_{k=1}^2\left(\delta_{kl}x+A_{N,k,l}\right)w_k,
\end{equation}
where $B_{N,l}(x)$ is a polynomial of degree $N-3$ and $A_{N,k,l}$
independent on $x$. Moreover, let $P^{I}_{N,l}(x)$ satisfies
\begin{equation*}
\int_{\mathbb{R}_+}P^{I}_{N,l}(x)x^j\D x=0,\quad j=0,\ldots,N-1.
\end{equation*}
Then the polynomials $B_{N,l}(x)$ and $\delta_{kl}x+A_{N,k,l}$ are
multi-orthogonal polynomials of type I with indices $(N,2,1)$ for
$l=1$ and $(N,1,2)$ for $l=2$. We will now show that $P^{II}_{N,1}$
and $P^{II}_{N,2}$ exist and are unique if both
$\left<L_{N-1},L_{N-2}\right>_1$ and
$\left<L_{N-3},L_{N-4}\right>_1$ are non-zero.
\begin{lemma}\label{le:exist}
Suppose both $\left<L_{N-1},L_{N-2}\right>_1$ and
$\left<L_{N-3},L_{N-4}\right>_1$ are non-zero, then the polynomials
$P^{II}_{N,1}$ and $P^{II}_{N,2}$ exist and are unique.
\end{lemma}
\begin{proof}
Let us write $L_{N-l}$ as
\begin{equation*}
L_{N-l}=\frac{d}{dx}\left(q_l(x)x(t-\tilde{\tau}x)w\right)w^{-1}+R_l(x),\quad
l=1,\ldots, 4.
\end{equation*}
for some polynomials $q_l$ of degree $N-l-2$ and $R_l$ of degree
$1$, then by Lemma \ref{le:linear}, we see that both $\varrho_{N-1}$
and $\varrho_{N-3}$ are invertible and hence the composition
$\varrho_{N-1}\varrho_{N-3}^{-1}$ is also invertible. In particular,
there exists a $2\times 2$ invertible matrix with entries $c_{l,k}$
and polynomials $\tilde{q}_l$ of degree $N-3$ such that
\begin{equation*}
\begin{split}
L_{N-l}=\frac{d}{dx}\left(\tilde{q}_l(x)x(t-\tilde{\tau}x)w\right)w^{-1}+c_{l-2,1}L_{N-1}+c_{l-2,2}L_{N-2},\quad
l=3,4.
\end{split}
\end{equation*}
Then by the first condition in (\ref{eq:mop1}), we see that
\begin{equation}\label{eq:ex}
\int_{\mathbb{R}_+}P^{II}_{N,l}w_j(x)dx=\left<P^{II}_{N,l},c_{l,1}L_{N-1}+c_{l,2}L_{N-2}\right>_1=-2\pi
i\delta_{lj},\quad l=1,2,\quad j=1,2.
\end{equation}
As the matrix with entries $c_{ij}$ is invertible, we see that the
linear equations (\ref{eq:ex}) has a unique solution in the linear
span of $L_{N-1}$ and $L_{N-2}$ if and only if
$\left<L_{N-1},L_{N-2}\right>_1\neq 0$.
\end{proof}
As we shall see, existence and uniqueness of $P^{II}_{N,l}$ would
imply that the multi-orthogonal polynomials of type I also exist and
are unique. As in \cite{Pierce}, the multi-orthogonal polynomial
together with the skew orthogonal polynomials form the solution of a
Riemann-Hilbert problem. Let $Y(x)$ be the matrix
\begin{equation}\label{eq:Ymatr}
Y(x)=\begin{pmatrix}\pi_{N,1}(x)&C\left(\pi_{N,1}w_0\right)&C\left(\pi_{N,1}w_1\right)
&C\left(\pi_{N,1}w_{2}\right)\\
\kappa\pi_{N-2,1}(x)&\cdots&&\cdots&\\
P^{II}_{N,1}(x)&\ddots&&\cdots&\\
P^{II}_{N,2}(x)&\cdots&&\cdots
\end{pmatrix},
\end{equation}
where $\kappa$ is the constant
\begin{equation}\label{eq:kappa}
\kappa=-\frac{2\pi i}{
\left<\pi_{N-2,1},x^{N-3}\right>_2}=-\frac{4\pi
i}{\tilde{\tau}Mh_{N-1,1}},\quad
h_{2j-1,1}=\left<\pi_{2j-2,1},\pi_{2j-1,1}\right>_1.
\end{equation}
and $C(f)$ is the Cauchy transform
\begin{equation}\label{eq:cauchy}
C(f)(x)=\frac{1}{2\pi i}\int_{\mathbb{R}_+}\frac{f(s)}{s-x}ds.
\end{equation}
Then by using the orthogonality conditions of the skew orthogonal
polynomials and multi-orthogonal polynomials, together with the jump
discontinuity of the Cauchy transform, one can check that $Y(x)$
satisfies the following Riemann-Hilbert problem.
\begin{equation}\label{eq:RHPY}
\begin{split}
1.\quad &\text{$Y(z)$ is analytic in
$\mathbb{C}\setminus\mathbb{R}_+$},\\
2.\quad &Y_+(z)=Y_-(z)\begin{pmatrix}1&w_0&w_1(z)&w_{2}(z)\\
0&1&0&0\\
0&0&1&0\\
0&0&0&1
\end{pmatrix},\quad z\in\mathbb{R}_+,\\
3.\quad &Y(z)=\left(I+O(z^{-1})\right)\begin{pmatrix}z^{N}\\
&z^{-N+2}\\
&&z^{-1}\\
&&&z^{-1}
\end{pmatrix},\quad z\rightarrow\infty.
\end{split}
\end{equation}
The multi-orthogonal polynomials of type I can also be arranged to
satisfy a Riemann-Hilbert problem. Let $\epsilon$ be the operator
\begin{equation}\label{eq:epsilon}
\epsilon(f)(x)=\frac{1}{2}\int_{0}^{\infty}\epsilon(x-y)f(y)\D y.
\end{equation}
First note that the functions $\psi_{j}(x)=\epsilon(\pi_{j,1}w)$ for
$j\geq 2$, can be express in the form of (\ref{eq:typeI}). By Lemma
\ref{le:linear}, we can write $\pi_{j,1}(x)$ as
\begin{equation*}
\pi_{j,1}(x)=\frac{d}{dx}\left(B_j(x)x(t-\tilde{\tau}x)w\right)w^{-1}+A_{j,1}L_{N-3}+A_{j,2}L_{N-4},
\end{equation*}
Then we have
\begin{equation}\label{eq:pitypeI}
\epsilon(\pi_{j,1}w)w=B_j(x)w_0+A_{j,1}w_1+A_{j,2}w_2.
\end{equation}
where $B_j(x)$ is a polynomial of degree $j-2$.

Let $X(z)$ be the matrix value function defined by
\begin{equation}\label{eq:Xmat}
X(z)=\begin{pmatrix}-\frac{\kappa\tilde{\tau}M}{2}
C\left(\psi_{N-2}w\right)&\frac{\kappa\tilde{\tau}M}{2}B_{N-2}&\frac{\kappa\tilde{\tau}M}{2}
A_{N-2,1}&\frac{\kappa\tilde{\tau}M}{2}A_{N-2,2}\\
-\frac{\tilde{\tau}M}{2}C\left(\psi_{N}w\right)&\frac{\tilde{\tau}M}{2}B_{N}&\frac{\tilde{\tau}M}{2}A_{N,1}&\frac{\tilde{\tau}M}{2}
A_{N,2}\\
-C\left(P^{I}_{N,1}\right)&\cdots&\cdots&\\
-C\left(P^{I}_{N,2}\right)&\cdots&\cdots
\end{pmatrix}.
\end{equation}
Then by using the orthogonality and the the jump discontinuity of
the Cauchy transform, it is easy to check that $X^{-T}(z)$ and
$Y(z)$ satisfies the same Riemann-Hilbert problem and hence the
multi-orthogonal polynomials of type I also exist and are unique.

We will now show that the kernel $S_1(x,y)$ given by (\ref{eq:ker})
can be expressed in terms of the matrix $Y(z)$.
\begin{proposition}\label{pro:CD}
Suppose
$\left<L_{N-3},L_{N-4}\right>_1\left<L_{N-1},L_{N-2}\right>_1\neq 0
$ and let the kernel $S_1(x,y)$ be
\begin{equation}\label{eq:kerS}
S_1(x,y)=-\sum_{j,k=0}^{N-1}r_j(x)w(x)\mu_{jk}\epsilon(r_kw)(y),
\end{equation}
where $r_j(x)$ is an arbitrary degree $j$ monic polynomial for
$j<N-2$ and $r_j(x)=\pi_{j,1}(x)$ for $j\geq N-2$. The matrix $\mu$
with entries $\mu_{jk}$ is the inverse of the matrix $\mathcal{M}$
whose entries are given by
\begin{equation}\label{eq:M}
\left(\mathcal{M}\right)_{jk}=\left<r_j,r_k\right>_1,\quad
j,k=0,\ldots,N-1.
\end{equation}
Then the kernel $S_1(x,y)$ exists and is equal to
\begin{equation}\label{eq:kerRHP}
S_{1}(x,y)=\frac{w(x)w^{-1}(y)}{2\pi i(x-y)}\begin{pmatrix}0&w_0(y)&
w_1(y)&w_{2}(y)\end{pmatrix}Y_+^{-1}(y)Y_+(x)\begin{pmatrix}1\\0\\0\\0\end{pmatrix}.
\end{equation}
\end{proposition}
\begin{proof} First note that, since $\left<L_{N-1},L_{N-2}\right>_1\neq 0$ and
$\left<L_{N-3},L_{N-4}\right>_1\neq 0$, the skew orthogonal
polynomials $\pi_{N-l,1}$ exist for $l=-1,\ldots,2$. In particular,
the moment matrix $\tilde{\mathcal{M}}$ with entries
\begin{equation*}
\left(\tilde{\mathcal{M}}\right)_{jk}=\left<x^j,y^k\right>_1, \quad
j,k=0,\ldots,N-1
\end{equation*}
is invertible. Since the polynomials $\pi_{N-j,1}$ exist for
$j=-1,\ldots,2$, the sequence $r_k(x)$ and $x^k$ are related by an
invertible transformation. Therefore the matrix $\mathcal{M}$ in
(\ref{eq:M}) is also invertible. As the matrix $\mathcal{M}$ is of
the form
\begin{equation*}
\mathcal{M}=\begin{pmatrix} \mathcal{M}_2&0\\
0&h_{N-1,1}\mathcal{J}
\end{pmatrix},\quad \mathcal{J}=\begin{pmatrix}0&1\\-1&0\end{pmatrix}
\end{equation*}
where $\mathcal{M}_2$ has entries $\left<r_j,r_k\right>_1$ for $j,k$
from $0$ to $N-3$. From this, we see that the matrix $\mu_{jk}$ is
of the form
\begin{equation}\label{eq:mu}
\mu=\begin{pmatrix}
\mathcal{M}_2^{-1}&0&\\
0&-h_{N-1,1}^{-1}\mathcal{J}
\end{pmatrix}.
\end{equation}
As in \cite{KDaem}, let us now expand the functions $xr_j(x)$ and
$x\epsilon(r_jw)$.
\begin{equation}\label{eq:cd}
\begin{split}
xr_j(x)&=\sum_{k=0}^{N-1}c_{jk}r_k(x)+\delta_{N-1,j}\pi_{N,1}(x),\\
x\epsilon(r_jw)(x)&=\sum_{k=0}^{N-1}d_{jk}\epsilon(r_kw)+d_{j,N}\psi_{N}
+d_{j,N+1}\frac{P^I_{N,1}(x)}{w(x)}+d_{j,N+2}\frac{P^I_{N,2}(x)}{w(x)}.
\end{split}
\end{equation}
Then the coefficients $c_{jk}$ and $d_{jk}$ for $j,k=0,\ldots,N-1$
are given by
\begin{equation*}
\begin{split}
c_{jk}=\sum_{l=0}^{N-1}\left<xr_j,r_l\right>_1\mu_{lk},\quad
d_{jk}=\sum_{l=0}^{N-1}\mu_{kl}\left<xr_l,r_k\right>_1.
\end{split}
\end{equation*}
Therefore if we let $\mathcal{C}$ be the matrix with entries
$c_{jk}$, $j,k=0,\ldots N-1$ and $\mathcal{D}$ be the matrix with
entries $d_{jk}$ for $j,k=0,\ldots N-1$, then we have
\begin{equation}\label{eq:expand}
\mathcal{C}=\mathcal{M}_1\mu,\quad
\mathcal{D}=\left(\mu\mathcal{M}_1\right)^T,
\end{equation}
where $\mathcal{M}_1$ is the matrix with entries
$\left<xr_j,r_k\right>_1$ for $j,k=0,\ldots,N-1$. From
(\ref{eq:cd}), we obtain
\begin{equation}\label{eq:S1}
\begin{split}
&(y-x)S_1(x,y)=\sum_{j=0}^{N-1}\pi_{N,1}(x)w(x)\mu_{N-1,j}\epsilon\left(r_{j}w\right)(y)\\
&-\sum_{j,k=0}^{N-1}r_j(x)w(x)\mu_{jk}\left(d_{k,N}\psi_{N}(y)+d_{k,N+1}\frac{P^I_{N,1}(y)}{w(y)}
+d_{k,N+2}\frac{P^I_{N,2}(y)}{w(y)}\right)\\
&+r^T(x)w(x)\left(\mathcal{C}^T\mu-\mu\mathcal{D}\right)\epsilon(rw)(y),
\end{split}
\end{equation}
where $r(x)$ is the column vector with components $r_k(x)$. Now by
(\ref{eq:expand}), we see that
\begin{equation*}
\mathcal{C}^T\mu=\mu^T\mathcal{M}_1^T\mu,\quad
\mu\mathcal{D}=\mu\mathcal{M}_1^T\mu^T,
\end{equation*}
which are equal as $\mu^T=-\mu$. Now from the form of $\mu$ in
(\ref{eq:mu}), we see that
$\mu_{N-1,j}=\delta_{j,N-2}h_{N-1,1}^{-1}$. Hence we have
\begin{equation}\label{eq:first}
\sum_{j=0}^{N-1}\pi_{N,1}(x)w(x)\mu_{N-1,j}\epsilon\left(r_{j}w\right)(y)=h_{N-1,1}^{-1}\pi_{N,1}(x)w(x)\psi_{N-2}(y).
\end{equation}
Let us now consider the second term in (\ref{eq:S1}). As in
(\ref{eq:pitypeI}) we can write $\epsilon(r_kw)w$ as
\begin{equation*}
\epsilon(r_kw)w=q_k(x)+D_{k,1}w_1+D_{k,2}w_2,
\end{equation*}
Then from the form of $P_{N,j}^{I}$ in (\ref{eq:typeI}) and the
orthogonality condition (\ref{eq:mop1}), we see that the
coefficients $d_{k,N+l}$, $l=1,2$ are given by
\begin{equation*}
d_{k,N+l}=D_{k,l}=-\frac{1}{2\pi
i}\int_{\mathbb{R}_+}P_{N,l}^{II}(x)\epsilon(r_kw)w\D x.
\end{equation*}
For $k\neq N-1$, the polynomial $q_k$ is of degree less than or
equal to $N-4$, while $B_{N}$ is a polynomial of degree $N-2$,
therefore the coefficient $d_{k,N}$ is zero unless $k=N-1$. For
$k=N-1$, it is given by the leading coefficient of $q_{N-1}$ divided
by the leading coefficient of $B_{N}$. Since
\begin{equation*}
\pi_{N-1,1}(x)=\frac{d}{dx}\left(q_{N-1}(x)x(t-\tilde{\tau}x)w\right)w^{-1}+D_{k,1}L_{N-3}+D_{k,2}L_{N-4},
\end{equation*}
we see that both the leading coefficient of $q_{N-1}(x)$ and
 $B_{N}$ is $\frac{2}{M\tilde{\tau}}$. Hence $d_{k,N}$ is $\delta_{k,N-1}$. This gives us
\begin{equation*}
\sum_{k,j=0}^{N-1}r_j(x)w(x)\mu_{jk}d_{k,N}\psi_{N}(y)=-h_{N-1,1}^{-1}\pi_{N-2,1}(x)w(x)\psi_{N}(y).
\end{equation*}
To express the second term in (\ref{eq:S1}) in terms of the
multi-orthogonal polynomials, let us now express $P^{II}_{N,l}$ in
terms of the polynomials $r_k$. Let us write
$P^{II}_{N,l}=\sum_{j=0}^{N-1}a_jr_j(x)$. Then we have
\begin{equation*}
\begin{split}
\int_{\mathbb{R}_+}P^{II}_{N,l}(x)\epsilon(r_kw)w\D
x=\sum_{j=0}^{N-1}a_j\left(\mathcal{M}_1\right)_{jk},\quad
\sum_{k=0}^{N-1}\int_{\mathbb{R}_+}P^{II}_{N,l}(x)\epsilon(r_kw)w\D
x\mu_{kj}=a_j.
\end{split}
\end{equation*}
Hence $P^{II}_{N,l}(x)$ can be written as
\begin{equation*}
\begin{split}
P^{II}_{N,l}(x)=\sum_{k,j=0}^{N-1}\left(\int_{\mathbb{R}_+}P^{II}_{N,l}(x)\epsilon(r_kw)w\D
x\mu_{kj}\right)r_j(x)=-2\pi
i\sum_{k,j=0}^{N-1}d_{k,N+l}\mu_{kj}r_j(x)
\end{split}
\end{equation*}
Therefore the second term in (\ref{eq:S1}) is given by
\begin{equation*}
\begin{split}
&\sum_{j,k=0}^{N-1}r_j(x)w(x)\mu_{jk}\left(d_{k,N}\psi_{N}(y)+d_{k,N+1}\frac{P^I_{N,1}(y)}{w(y)}
+d_{k,N+2}\frac{P^I_{N,2}(y)}{w(y)}\right)\\
&=-h_{N-1,1}^{-1}\pi_{N-2,1}(x)\psi_{N}(y)+\frac{1}{2\pi
i}\sum_{l=1}^2P^{II}_{N,l}(x)w(x)w^{-1}(y)P^I_{N,l}(y).
\end{split}
\end{equation*}
From this and (\ref{eq:first}), we obtain
\begin{equation*}
\begin{split}
(y-x)S_1(x,y)&=h_{N-1,1}^{-1}\pi_{N,1}(x)w(x)\psi_{N-2}(y)+
h_{N-1,1}^{-1}\pi_{N-2,1}(x)w(x)\psi_{N}(y)\\
&-\frac{1}{2\pi
i}\sum_{l=1}^2P^{II}_{N,l}(x)w(x)w^{-1}(y)P^I_{N,l}(y).
\end{split}
\end{equation*}
By using the fact that $Y^{-1}(y)=X^T(y)$ and the expressions of the
matrix $Y$ (\ref{eq:Ymatr}) and $X$ (\ref{eq:Xmat}), together with
(\ref{eq:kappa}), we see that this is the same as (\ref{eq:kerRHP}).
\end{proof}
\subsection{The kernel in terms of Laguerre polynomials}
We will now use a result in \cite{baikext} to further simplify the
expression of the kernel $S_1(x,y)$ so that its asymptotics can be
computed using the asymptotics of Laguerre polynomials. Let us
recall the set up in \cite{baikext}. First let $Y(x)$ be a matrix
satisfying the Riemann-Hilbert problem
\begin{equation*}
\begin{split}
1.\quad &\text{$Y(z)$ is analytic in
$\mathbb{C}\setminus\mathbb{R}_+$},\\
2.\quad &Y_+(z)=Y_-(z)\begin{pmatrix}1&w_0(z)&w_1(z)&\cdots&w_{r}(z)\\
0&1&0&\cdots&0\\
\vdots&&\ddots&\vdots&\\
0&&\cdots&&1
\end{pmatrix},\quad z\in\mathbb{R}_+\\
3.\quad &Y(z)=\left(I+O(z^{-1})\right)\begin{pmatrix}z^{n}\\
&z^{-n+r}\\
&&z^{-1}\\
&&&\ddots\\
&&&&z^{-1}
\end{pmatrix},\quad z\rightarrow\infty.
\end{split}
\end{equation*}
and let $\mathcal{K}_1(x,y)$ be the kernel given by
\begin{equation}\label{eq:k1}
\mathcal{K}_1(x,y)=\frac{w_0(x)w_0^{-1}(y)}{2\pi
i(x-y)}\begin{pmatrix}0&w_0(y)&
\ldots&w_{r}(y)\end{pmatrix}Y_+^{-1}(y)Y_+(x)\begin{pmatrix}1\\0\\\vdots\\0\end{pmatrix}.
\end{equation}
Let $\pi_{j,2}(x)$ be the monic orthogonal polynomials with respect
to the weight $w_0(x)$. Let $\mathcal{K}_0(x,y)$ be the following
kernel.
\begin{equation}\label{eq:k0}
\mathcal{K}_0(x,y)=w_0(x)\frac{\pi_{2,n}(x)\pi_{2,n-1}(y)-\pi_{2,n}(y)\pi_{2,n-1}(x)}{h_{n-1,2}(x-y)},
\end{equation}
where $h_{j,2}=\int_{0}^{\infty}\pi_{j,2}^2w_0\D x$. Let $\pi(z)$,
$v(z)$ and $u(z)$ be the following vectors
\begin{equation}\label{eq:tv}
\pi(z)=\left(\pi_{n-r,2},\ldots,\pi_{n-1,2}\right)^T,\quad
v(z)=\left(w_1,\ldots,w_r\right)^Tw_0^{-1}, \quad
u(z)=\left(I-\mathcal{K}_0^T\right)v(z)
\end{equation}
and let $B$ be the matrix
$B=\int_{\mathbb{R}_+}\pi(z)v^T(z)w_0(z)dz$. Then the kernel
$\mathcal{K}_1(x,y)$ can be express as \cite{baikext}
\begin{proposition}\label{pro:baik}
Suppose $\int_{\mathbb{R}_+}p(x)w_i(x)\D x$ converges for any
polynomial $p(x)$. Then the kernel $\mathcal{K}_1(x,y)$ defined by
(\ref{eq:k1}) is given by
\begin{equation}\label{eq:kerform}
\mathcal{K}_1(x,y)-\mathcal{K}_0(x,y)=w_0(x)u^T(y)B^{-1}\pi(x).
\end{equation}
\end{proposition}
\begin{remark} Although in \cite{baikext}, the theorem is stated
with the jump of $Y(x)$ on $\mathbb{R}$ instead of $\mathbb{R}_+$,
while the weights are of the special form $w_0=e^{-NV(x)}$,
$w_j(x)=e^{a_jx}$, where $V(x)$ is an even degree polynomial, the
proof in \cite{baikext} in fact remain valid as long as integrals of
the form $\int_{\mathbb{R}_+}p(x)w_i(x)\D x$ converges for any
polynomial $p(x)$. This is true in our case.
\end{remark}
We can now apply Proposition \ref{pro:baik} to our case. In our
case, the vectors $\pi(z)$ and $v(z)$ are given by
\begin{equation*}
\pi(z)=\left(L_{N-2}\quad L_{N-1}\right)^T,\quad v(z)=\left(w_1\quad
w_2\right)^Tw_0^{-1},
\end{equation*}
while the matrix $B$ is given by
\begin{equation*}
B=\begin{pmatrix}\left<L_{N-2},L_{N-3}\right>_1&\left<L_{N-2},L_{N-4}\right>_1\\
\left<L_{N-1},L_{N-3}\right>_1&\left<L_{N-1},L_{N-4}\right>_1\end{pmatrix}
\end{equation*}
By corollary \ref{cor:linear}, we see that
$\left<L_{N-2},L_{N-3}\right>_1=0$ and hence the determinant of $B$
is
\begin{equation*}
\det B=\left<L_{N-2},L_{N-4}\right>_1\left<L_{N-1},L_{N-3}\right>_1.
\end{equation*}
From Lemma \ref{le:exist}, we see that $B$ is invertible if and only
if the multi-orthogonal polynomials $P^{II}_{N,l}$ exist. Let us now
consider the vector $u(z)$. It is given by
\begin{equation*}
u(z)=\left(I-\mathcal{K}_0^T\right)v(z)=v(z)-\sum_{j=0}^{N-1}\frac{L_j(z)}{h_{j,0}}\left(\left<L_j,L_{N-3}\right>_1\quad\left<L_j,L_{N-4}\right>_1\right)
\end{equation*}
by the Christoffel-Darboux formula, where
$h_{j,0}=\left<L_j,L_j\right>_2$. We shall show that $L_{N-3}$ and
$L_{N-4}$ can be written in the following form
\begin{equation}\label{eq:mappi}
L_{N-l}=\frac{d}{dx}\left(q_l(x)x(t-\tilde{\tau}x)w\right)w^{-1}+C_{l-2,1}\pi_{N+1,1}+C_{l-2,2}\pi_{N,1},\quad
l=3,4,
\end{equation}
for some polynomial $q_l(x)$ of degree $N-1$. By Lemma
\ref{le:linear}, we see that if $\left<L_{N-3},L_{N-4}\right>_1\neq
0$, then the map $\varrho_{N-3}$ in Definition \ref{de:fn} is
invertible. Therefore if the restriction of the map $f$ in
Definition \ref{de:fn} is also invertible on the span of
$\pi_{N+1,1}$ and $\pi_{N,1}$, we will be able to write $L_{N-3}$
and $L_{N-4}$ in the form of (\ref{eq:mappi}).
\begin{lemma}\label{le:pi2nmap}
Let $f$ be the map in Definition \ref{de:fn} and let $\varrho_{\pi}$
be its restriction to the span of $\pi_{N+1,1}$ and $\pi_{N,1}$.
Then $\varrho_{\pi}$ is invertible.
\end{lemma}
\begin{proof} Suppose there exist $a_1$ and $a_2$ such that
\begin{equation*}
a_1\pi_{N+1,1}+a_2\pi_{N,1}=\frac{d}{dx}\left(q(x)x(t-\tilde{\tau}x)w\right)w^{-1}
\end{equation*}
for some polynomial $q(x)$ of degree at most $N-1$. Then we have
\begin{equation*}
\left<x^j,a_1\pi_{N+1,1}+a_2\pi_{N,1}\right>_1=\left<x^jq(x)\right>_2=0,\quad
j=0,\ldots,N-1.
\end{equation*}
As the degree of $q(x)$ is at most $N-1$, this is only possible if
$q(x)=0$.
\end{proof}
The composition of $\varrho_{N-3}$ and $\varrho_{\pi}^{-1}$ will
therefore give us a representation of $L_{N-3}$ and $L_{N-4}$ in the
form of (\ref{eq:mappi}). By using this representation and the fact
that $q_l(x)$ is of degree at most $N-1$, we see that
\begin{equation*}
\begin{split}
\mathcal{K}_0^T\left(w_lw_0^{-1}\right)&=\sum_{j=0}^{N-1}\frac{L_j(x)}{h_{j,0}}\left<L_j,L_{N-l-2}\right>_1\\
&=\sum_{j=0}^{N-1}\frac{L_j(x)}{h_{j,0}}\left(\left<L_j,q_l\right>_2+\left<L_j,C_{l-2,1}\pi_{N+1,1}+C_{l-2,2}\pi_{N,1}\right>_1\right),\\
&=q_l(x).
\end{split}
\end{equation*}
Therefore the vector $u(x)$ is given by
\begin{equation*}
u(x)=w(x)w_0^{-1}(x)C\epsilon\left(\pi_{N+1,1}w\quad
\pi_{N,1}w\right)^T
\end{equation*}
where $C$ is the matrix with entries $C_{i,j}$. We will now
determine the constants $C_{i,j}$.
\begin{lemma}\label{le:A}
Let $\pi_{N,1}$ and $\pi_{N+1,1}$ be the monic skew orthogonal
polynomial with respect to the weight $w(x)$ and choose
$\pi_{N+1,1}$ so that the constant $c$ in (\ref{eq:p2N}) is zero.
Then the vector $u(y)$ in (\ref{eq:kerform}) is given by
\begin{equation}\label{eq:ux}
u(x)=w(x)w_0^{-1}(x)C\epsilon\left(\pi_{N+1,1}w\quad
\pi_{N,1}w\right)^T,
\end{equation}
where $C$ is the matrix whose entries $C_{i,j}$ are given by
\begin{equation}\label{eq:Aent}
\begin{split}
C_{i,1}&=-\frac{M\tilde{\tau}\left<L_{N-1},L_{N-i-2}\right>_1}{2h_{N-1,0}},\\
C_{i,2}&=\left(Mt-\tilde{\tau}\left(N+M\right)\right)\frac{\left<L_{N-1},L_{N-i-2}\right>_1}{2h_{N-1,0}}-M\tilde{\tau}
\frac{\left<L_{N-2},L_{N-i-2}\right>_1}{2h_{N-2,0}}.
\end{split}
\end{equation}
\end{lemma}
\begin{proof} First let us compute the leading order coefficients of
the polynomial $q_l(x)$ in (\ref{eq:mappi}). Let
$q_l(x)=q_{l,N-1}x^{N-1}+q_{l,N-2}x^{N-2}+O(x^{N-3})$, then we have
\begin{equation*}
\begin{split}
&\frac{d}{dx}\left(q_l(x)x(t-\tilde{\tau}x)w\right)w^{-1}=\frac{M\tilde{\tau}}{2}q_{l,N-1}x^{N+1},\\
&+\left(-\frac{\tilde{\tau}}{2}(N+M)q_{l,N-1}-\frac{Mt}{2}q_{l,N-1}+\frac{M\tilde{\tau}}{2}q_{l,N-2}\right)x^{N}+O(x^{N-1})
\end{split}
\end{equation*}
From (\ref{eq:mappi}), we see that
\begin{equation}\label{eq:qlead}
q_{l,N-1}=-\frac{2}{M\tilde{\tau}}C_{l-2,1}.
\end{equation}
On the other hand, by orthogonality, we have
\begin{equation*}
\left<L_{N-1},L_{N-l}\right>_1=\left<L_{N-1},q_l\right>_2=q_{l,N-1}h_{N-1,0}.
\end{equation*}
Therefore $C_{l-2,1}$ is given by
\begin{equation*}
C_{l-2,1}=-\frac{M\tilde{\tau}\left<L_{N-1},L_{N-l}\right>_1}{2h_{N-1,0}}
\end{equation*}
Let us now compute $C_{i,2}$. By taking the skew product, we have
\begin{equation*}
\left<L_{N-2},L_{N-l}\right>_1=q_{l,N-1}\left<L_{N-2},x^{N-1}\right>_2
+q_{l,N-2}h_{N-2,0}
\end{equation*}
Now from (\ref{eq:laguerre}), we obtain
\begin{equation}\label{eq:nnp1}
\begin{split}
\left<x^{j+1},L_{j}\right>_2&=\left<L_{j+1}+\frac{(M-N+j+1)(j+1)}{M}x^j,L_j\right>_2,\\
&=\frac{(M-N+j+1)(j+1)}{M}h_{j,0}.
\end{split}
\end{equation}
Hence $q_{l,N-2}$ is equal to
\begin{equation}\label{eq:q2}
q_{l,N-2}=\frac{\left<L_{N-2},L_{N-l}\right>_1}{h_{N-2,0}}-q_{l,N-1}\frac{(M-1)(N-1)}{M}
\end{equation}
By using the expansion (\ref{eq:p2N}) of the skew orthogonal
polynomials in terms of $L_{N-k}$, we have
\begin{equation*}
\begin{split}
\left<L_{N-l},L_{N}\right>_2&=\frac{M\tilde{\tau}}{2}q_{l,N-1}\left<x^{N+1},L_{N}\right>_2+C_{l-2,2}h_{N,0}\\
&+\left(\frac{\tilde{\tau}}{2}(-N-M)q_{l,N-1}-\frac{Mt}{2}q_{l,N-1}+\frac{M\tilde{\tau}}{2}q_{l,N-2}\right)h_{N,0}
\end{split}
\end{equation*}
By substituting (\ref{eq:q2}) into this, we obtain
\begin{equation*}
C_{l-2,2}=\left(-\frac{\tilde{\tau}}{2}\left(M+N\right)+\frac{Mt}{2}\right)\frac{\left<L_{N-1},L_{N-l}\right>_1}{h_{N-1,0}}-\frac{M\tilde{\tau}}{2}
\frac{\left<L_{N-2},L_{N-l}\right>_1}{h_{N-2,0}}
\end{equation*}
This proves the lemma.
\end{proof}
From this and (\ref{eq:kerform}), we obtain the following.
\begin{corollary}\label{cor:baik}
The kernel $S_1(x,y)$ defined by (\ref{eq:kerS}) is given by
\begin{equation}\label{eq:kerform1}
\begin{split}
&S_1(x,y)-K_2(x,y)=\\
&\epsilon\left(\pi_{N+1,1}w\quad\pi_{N,1}w\right)(y)\begin{pmatrix}0&-\frac{M\tilde{\tau}}{2h_{N-1}}\\
-\frac{M\tilde{\tau}}{2h_{N-2}}&\frac{Mt-\tilde{\tau}(N+M)}{2h_{N-1}}\end{pmatrix}\pi(x)w(x)
\end{split}
\end{equation}
where $K_2(x,y)$ is the kernel of the Laguerre polynomials
\begin{equation}\label{eq:k2}
K_2(x,y)=\left(\frac{y(t-\tilde{\tau}y)}{x(t-\tilde{\tau}x)}\right)^{\frac{1}{2}}
w_0^{\frac{1}{2}}(x)w_0^{\frac{1}{2}}(y)\frac{L_{N}(x)L_{N-1}(y)-L_N(y)L_{N-1}(x)}{h_{N-1,0}(x-y)}
\end{equation}
\end{corollary}
\section{Derivative of the partition function}
In this section we will derive a formula for the derivative of
determinant of the matrix $\mathcal{M}$ given in (\ref{eq:M}). We
have the following.
\begin{proposition}\label{pro:derpar}
Let $\mathcal{M}$ be the matrix given by (\ref{eq:M}), where the
sequence of monic polynomials $r_j(x)$ in (\ref{eq:M}) is chosen
such that $r_j(x)$ are arbitrary degree $j$ monic polynomials that
are independent on $t$ and $r_j(x)=\pi_{j,1}(x)$ for $j=N-2,N-1$.
Then the logarithmic derivative of $\det\mathcal{M}$ with respect to
$t$ is given by
\begin{equation}\label{eq:derpar}
\frac{\p}{\p
t}\log\det\mathcal{M}=\int_{\mathbb{R}_+}\frac{S_1(x,x)}{t-\tilde{\tau}x}\D
x,
\end{equation}
where $S_1(x,y)$ is the kernel given in (\ref{eq:kerS}).
\end{proposition}
\begin{proof} First let us differential the determinant
$\det\mathcal{M}$ with respect to $t$. We have
\begin{equation*}
\begin{split}
\frac{\p}{\p
t}\det\mathcal{M}&=\det\begin{pmatrix}\p_tM_{00}&M_{01}&\cdots&M_{0,N-1}\\
\vdots&\vdots&\ddots&\vdots\\
\p_tM_{2n-1,0}&M_{2n-1,1}&\cdots&M_{N-1,N-1}\end{pmatrix}\\&+
\det\begin{pmatrix}M_{00}&\p_tM_{01}&\cdots&M_{0,N-1}\\
\vdots&\vdots&\ddots&\vdots\\
M_{2n-1,0}&\p_tM_{2n-1,1}&\cdots&M_{N-1,N-1}\end{pmatrix}\\
&+\cdots+
\det\begin{pmatrix}M_{00}&M_{01}&\cdots&\p_tM_{0,N-1}\\
\vdots&\vdots&\ddots&\vdots\\
M_{N-1,0}&M_{N-1,1}&\cdots&\p_tM_{N-1,N-1}\end{pmatrix}.
\end{split}
\end{equation*}
Computing the individual determinants using the Laplace formula, we
obtain
\begin{equation}\label{eq:logder}
\frac{\p}{\p
t}\det\mathcal{M}=\det\mathcal{M}\sum_{i,j=0}^{N-1}\p_tM_{ij}\mu_{ji}.
\end{equation}
As $r_j(x)$ are independent on $t$ for $j<N-2$, the derivative
$\p_tM_{ij}$ is given by
\begin{equation*}
\p_tM_{ij}=-\frac{1}{2}\left(\left<\frac{r_i}{t-\tilde{\tau}x},r_j\right>_1+
\left<r_i,\frac{r_j}{t-\tilde{\tau}y}\right>_1\right)=
-\frac{1}{2}\left(\left<\frac{r_i}{t-\tilde{\tau}x},r_j\right>_1-
\left<\frac{r_j}{t-\tilde{\tau}x},r_i\right>_1\right)
\end{equation*}
for $i,j<N-2$. For either $i$ or $j$ equal to $N-2$ or $N-1$, we
have
\begin{equation*}
\begin{split}
\p_tM_{i,N-1}&=\delta_{N-2,i}\Bigg(-\frac{1}{2}\left(\left<\frac{\pi_{N-2,1}}{t-\tilde{\tau}{x}},\pi_{N-1,1}\right>_1
-\left<\frac{\pi_{N-1,1}}{t-\tilde{\tau}{x}},\pi_{N-2,1}\right>_1\right)\\
&+\left<\p_t\pi_{N-2,1},\pi_{N-1,1}\right>_1+\left<\pi_{N-2,1},\p_t\pi_{N-1,1}\right>_1
\Bigg).
\end{split}
\end{equation*}
Note that by orthogonality, the last two terms in the above
expression are zero as $\p_t\pi_{N-1,1}$ is of degree $N-2$ and
$\p_t\pi_{N-2,1}$ is of degree $N-3$. Applying the same argument to
$\p_tM_{i,N-2}$, we obtain
\begin{equation*}
\begin{split}
\p_tM_{i,N-1}&=-\delta_{N-2,i}\left(\left<\frac{\pi_{N-2,1}}{t-\tilde{\tau}{x}},\pi_{N-1,1}\right>_1
-\left<\frac{\pi_{N-1,1}}{t-\tilde{\tau}{x}},\pi_{N-2,1}\right>_1\right)\\
\p_tM_{i,N-2}&=-\frac{\delta_{N-1,i}}{2}\left(\left<\frac{\pi_{N-1,1}}{t-\tilde{\tau}{x}},\pi_{N-2,1}\right>_1
-\left<\frac{\pi_{N-2,1}}{t-\tilde{\tau}{x}},\pi_{N-1,1}\right>_1\right).
\end{split}
\end{equation*}
As $\mathcal{M}$ is anti-symmetric, the derivatives $\p_tM_{N-1,i}$
and $\p_tM_{N-2,i}$ are given by $\p_tM_{N-1,i}=-\p_tM_{i,N-1}$ and
$\p_tM_{N-2,i}=-\p_tM_{i,N-2}$. From these and (\ref{eq:logder}), we
obtain
\begin{equation}\label{eq:de}
\begin{split}
\frac{\p}{\p
t}\det\mathcal{M}&=\det\mathcal{M}\Bigg(\sum_{i,j=0}^{N-3}\left<\frac{r_i}{t-\tilde{\tau}x},r_j\right>_1\mu_{ij}
\\&+2\left(\left<\frac{\pi_{N-2,1}}{t-\tilde{\tau}{x}},\pi_{N-1,1}\right>_1
-\left<\frac{\pi_{N-1,1}}{t-\tilde{\tau}{x}},\pi_{N-2,1}\right>_1\right)\mu_{N-2,N-1}\Bigg),
\end{split}
\end{equation}
where we have used the anti-symmetry of $\mathcal{M}$ and $\mu$ to
obtain the last term. From the structure of the matrix $\mu$ in
(\ref{eq:mu}), we see that the last term in (\ref{eq:de}) can be
written as
\begin{equation*}
\begin{split}
&2\left(\left<\frac{\pi_{N-2,1}}{t-\tilde{\tau}{x}},\pi_{N-1,1}\right>_1
-\left<\frac{\pi_{N-1,1}}{t-\tilde{\tau}{x}},\pi_{N-2,1}\right>_1\right)\mu_{N-2,N-1}\Bigg)=\\
&\sum_{i=0}^{N-1}\sum_{j=N-2}^{N-1}\left<\frac{r_i}{t-\tilde{\tau}x},r_j\right>_1\mu_{ij}
+\sum_{i=N-2}^{N-1}\sum_{j=0}^{N-1}\left<\frac{r_i}{t-\tilde{\tau}x},r_j\right>_1\mu_{ij}
\end{split}
\end{equation*}
From this, (\ref{eq:de}) and the expression of the kernel in
(\ref{eq:kerS}), we obtain (\ref{eq:derpar}).
\end{proof}
\section*{Appendix: A proof of the j.p.d.f. formula using Zonal polynomials}
\renewcommand{\theequation}{A.\arabic{equation}}
\setcounter{equation}{0}

We present here a simpler algebraic proof of Theorem \ref{thm:main1}
using Zonal polynomials. Zonal polynomials are introduced by James
\cite{James} and Hua \cite{Hu} independently. They are polynomials
with matrix argument that depend on an index $p$ which is a
partition of an integer $k$. The real Zonal polynomials $Z_p(X)$
take arguments in symmetric matrices and are homogenous polynomials
in the eigenvalues of its matrix argument $X$. We shall not go into
the details of their definitions, but only state the important
properties of these polynomials that is relevant to our proof.
Readers who are interested can refer to the excellent references of
\cite{Muir}, \cite{Mac} and \cite{Ta}.

Let $p$ be a partition of an integer $k$ and let $l(p)$ be the
length of the partition. We will use $p\vdash k$ to indicate that
$p$ is a partition of $k$. Let $X$ and $Y$ be $N\times N$ symmetric
matrices and $x_i$, $y_i$ their eigenvalues. Given a partition
$p=(p_1,\ldots,p_{l(p)})$ of the integer $k$, we will order the
parts $p_i$ such that if $i<j$, then $p_i>p_j$. If we have 2
partitions $p$ and
$p^{\prime}=(p^{\prime}_1,\ldots,p_{l(p^{\prime})}^{\prime})$, then
we say that $p<p^{\prime}$ if there exists an index $j$ such that
$p_i=p_i^{\prime}$ for $i<j$ and $p_j<p_j^{\prime}$. Let the
monomial $x^p$ be $x_1^{p_1}\ldots x_{p_{l(p)}}^{p_{l(p)}}$, then we
say that $x^{p^{\prime}}$ is of a higher weight than $x^p$ if
$p^{\prime}>p$. Then the Zonal polynomial $Z_{p}(X)$ is a homogenous
polynomial of degree $k$ in the eigenvalues $x_j$ with the highest
weight term being $x^p$. It has the following properties.
\begin{equation}\label{eq:zonal}
\begin{split}
&\left(\tr(X)\right)^k=\sum_{p\vdash k}Z_p(X),\\
&\int_{O(N)}e^{-My\tr\left(XgYg^T\right)}\D g
=\sum_{k=0}^{\infty}\frac{(My)^k}{k!}\sum_{p\vdash
k}\frac{Z_p(X)Z_p(Y)}{Z_p(I_N)}
\end{split}
\end{equation}
These properties can be found in the references \cite{Muir},
\cite{Mac} and \cite{Ta}. Another important property is the
following generating function formula for the Zonal polynomials,
which can be found in \cite{Mac} and \cite{Ta}.
\begin{equation}\label{eq:genfun}
\prod_{i,j=1}^N\left(1-2\theta
x_iy_j\right)^{-\frac{1}{2}}=\sum_{k=0}^{\infty}\frac{\theta^k}{k!}\sum_{p\vdash
k}\frac{Z_p(X)Z_p(Y)}{d_p}
\end{equation}
for some constant $d_p$. In particular, if $(k)$ is the partition of
$k$ with length 1, that is, $(k)=(k,0,\ldots,0)$, then the constant
$d_{(k)}$ is given by
\begin{equation*}
d_{(k)}=\frac{1}{(2k-1)!!}.
\end{equation*}
For the rank 1 spiked model, let us consider the case where all but
one $y_j$ is zero and denote the non-zero eigenvalue by $y$. Then
from the fact that the highest weight term in $Z_p(Y)$ is
$y_1^{p_1}\ldots y_{p_{l(p)}}^{p_{l(p)}}$, we see that the only
non-zero $Z_p(Y)$ is $Z_{(k)}(Y)$, which by the first equation in
(\ref{eq:zonal}), is simply $y^k$. Therefore the formulae in
(\ref{eq:zonal}) and (\ref{eq:genfun}) are greatly simplified in
this case.
\begin{equation}\label{eq:gen2}
\begin{split}
&\int_{O(N)}e^{-My\tr\left(XgYg^T\right)}\D g
=\sum_{k=0}^{\infty}(My)^k\frac{Z_{(k)}(X)y^k}{k!Z_{(k)}(I_N)},\\
&\prod_{i=1}^N\left(1-2\theta
x_iy\right)^{-\frac{1}{2}}=\sum_{k=0}^{\infty}\theta^k\frac{(2k-1)!!Z_{(k)}(X)y^k}{k!}
\end{split}
\end{equation}
By using the generating function formula, we see that $Z_{(k)}(I_N)$
is given by
\begin{equation}\label{eq:ZI}
Z_{(k)}(I_N)=\frac{\left(N/2+k-1\right)!2^k}{(N/2-1)!(2k-1)!!}.
\end{equation}
By taking $\theta=\frac{1}{2t}$ in the second equation of
(\ref{eq:gen2}), we see that
\begin{equation*}
\prod_{i=1}^N\left(t-
x_iy\right)^{-\frac{1}{2}}=t^{-\frac{N}{2}}\sum_{k=0}^{\infty}(2t)^{-k}\frac{(2k-1)!!Z_{(k)}(X)y^k}{k!}.
\end{equation*}
We can now compute the integral
\begin{equation*}
S(t)=\int_{\Gamma}e^{Mt}\prod_{i=1}^N\left(t-
x_iy\right)^{-\frac{1}{2}}\D t
\end{equation*}
by taking residue at $\infty$, which is the $t^{-1}$ coefficient in
the following expansion
\begin{equation*}
e^{Mt}\prod_{i=1}^N\left(t-
x_iy\right)^{-\frac{1}{2}}=\sum_{k,j=0}^{\infty}\frac{M^jt^{-\frac{N}{2}+j-k}(2k-1)!!Z_{(k)}(X)y^k}{2^kj!k!}.
\end{equation*}
This coefficient is given by
\begin{equation*}
\begin{split}
S(t)&=M^{\frac{N}{2}-1}\sum_{k=0}^{\infty}\frac{Z_{(k)}(X)(2k-1)!!y^kM^k}{2^k(N/2+k-1)!k!}\\
&=\frac{M^{\frac{N}{2}-1}}{\left(N/2-1\right)!}\sum_{k=0}^{\infty}\frac{Z_{(k)}(X)y^kM^k}{Z_{(k)}(I_N)k!}=\frac{M^{\frac{N}{2}-1}}{\left(N/2-1\right)!}\int_{O(N)}e^{-My\tr\left(XgYg^T\right)}\D
g.
\end{split}
\end{equation*}
This proves Theorem \ref{thm:main1}. There also exist complex and
quarternionic Zonal polynomials $C_{p}(X)$ and $Q_{p}(X)$ which
satisfy the followings instead.
\begin{equation*}
\begin{split}
&\int_{U(N)}e^{-My\tr\left(XgYg^{\dag}\right)}g^{\dag}\D g
=\sum_{k=0}^{\infty}\frac{(My)^k}{k!}\sum_{p\vdash
k}\frac{C_p(X)C_p(Y)}{C_p(I_N)},\\
&\int_{Sp(N)}e^{-My\mathrm{Re}\left(\tr\left(XgYg^{-1}\right)\right)}g^{-1}\D
g =\sum_{k=0}^{\infty}\frac{(My)^k}{k!}\sum_{p\vdash
k}\frac{Q_p(X)Q_p(Y)}{Q_p(I_N)}.
\end{split}
\end{equation*}
Their generating functions are given by
\begin{equation*}
\begin{split}
&\prod_{i,j=1}^N\left(1-2\theta
x_iy_j\right)^{-1}=\sum_{k=0}^{\infty}\frac{\theta^k}{k!}\sum_{p\vdash
k}\frac{C_p(X)C_p(Y)}{c_p},\\
&\prod_{i,j=1}^N\left(1-2\theta
x_iy_j\right)^{-2}=\sum_{k=0}^{\infty}\frac{\theta^k}{k!}\sum_{p\vdash
k}\frac{Q_p(X)Q_p(Y)}{q_p}
\end{split}
\end{equation*}
where $c_{(k)}$ and $q_{(k)}$ are
\begin{equation*}
\begin{split}
c_{(k)}=\frac{1}{2^kk!},\quad q_{(k)}=\frac{1}{(k+1)!2^k}
\end{split}
\end{equation*}
Then by following the same argument as in the real case, we can
write down the following integral formulae for rank one
perturbations of the complex and quarternionic cases.
\begin{equation}\label{eq:compsym}
\begin{split}
&\int_{U(N)}e^{-My\tr\left(XgYg^{\dag}\right)}g^{\dag}\D g
=\frac{(N-1)!}{M^{N-1}}\int_{\Gamma}e^{Mt}\prod_{i=1}^N\left(t-
x_iy\right)^{-1}\D t,\\
&\int_{Sp(N)}e^{-My\mathrm{Re}\left(\tr\left(XgYg^{-1}\right)\right)}g^{-1}\D
g =\frac{(2N-1)!}{M^{2N-1}}\int_{\Gamma}e^{Mt}\prod_{i=1}^N\left(t-
x_iy\right)^{-2}\D t.
\end{split}
\end{equation}

\vspace{.25cm}

\noindent\rule{16.2cm}{.5pt}

\vspace{.25cm}

{\small

\noindent {\sl School of Mathematics \\
                       University of Bristol\\
                       Bristol BS8 1TW, UK  \\
                       Email: {\tt m.mo@bristol.ac.uk}

                       \vspace{.25cm}

                       \noindent  24 November  2010}}
\end{document}